\documentclass[11pt,letterpaper]{amsart}
\usepackage{graphicx}
\usepackage{hyperref}
\usepackage{tikz}
\usepackage{xcolor}
\usepackage{comment}
\usepackage{mathrsfs}
\usepackage[T1]{fontenc}
\usepackage{xypic}

\usepackage{amsmath,amsthm,amssymb,amscd,enumerate}

\makeatletter
\def\@evenhead{\footnotesize\thepage \hfil M.~Kawasaki, M.~Kimura, H.~Kodama, Y.~Matsuda, T.~Matsushita, and R.~Orita \hfil }
\makeatother

\newtheorem{thm}{Theorem}[section]
\newtheorem{theorem}[thm]{Theorem}
\newtheorem{prop}[thm]{Proposition}
\newtheorem{proposition}[thm]{Proposition}
\newtheorem{lem}[thm]{Lemma}
\newtheorem{lemma}[thm]{Lemma}

\newtheorem{corollary}[thm]{Corollary}
\newtheorem{prob}[thm]{Problem}

\theoremstyle{definition}
\newtheorem{definition}[thm]{Definition}
\newtheorem{example}[thm]{Example}

\theoremstyle{remark}

\newtheorem{remark}[thm]{Remark}

\numberwithin{equation}{section}

\allowdisplaybreaks[1]

\newcommand{\QQ}{\mathbb{Q}}
\newcommand{\RR}{\mathbb{R}}
\newcommand{\ZZ}{\mathbb{Z}}
\newcommand{\NN}{\mathbb{N}}

\newcommand{\ee}{\epsilon}

\newcommand{\Ham}{\mathrm{Ham}}
\newcommand{\Diff}{\mathrm{Diff}^c}
\newcommand{\tHam}{\widetilde{\mathrm{Ham}}}

\newcommand{\supp}{\mathrm{supp}}
\newcommand{\Flux}{\mathrm{Flux}}
\newcommand{\tFlux}{\widetilde{\mathrm{Flux}}}

\newcommand{\Cal}{\mathrm{Cal}}
\newcommand{\bCal}{\underline{\mathrm{Cal}}}

\newcommand{\Ker}{\operatorname{Ker}}
\newcommand{\tKer}{\widetilde{\operatorname{Ker}}}
\newcommand{\Id}{\mathrm{Id}}

\newcommand{\TK}{\mathcal{M}}

\newcommand{\dg}{d}
\newcommand{\du}{\bar{d}}
\newcommand{\ts}{\tilde{s}}

\newcommand{\relmiddle}[1]{\mathrel{}\middle#1\mathrel{}}

\makeatletter
\newsavebox{\@brx}
\newcommand{\llangle}[1][]{\savebox{\@brx}{\(\m@th{#1\langle}\)}%
  \mathopen{\copy\@brx\kern-0.5\wd\@brx\usebox{\@brx}}}
\newcommand{\rrangle}[1][]{\savebox{\@brx}{\(\m@th{#1\rangle}\)}%
  \mathclose{\copy\@brx\kern-0.5\wd\@brx\usebox{\@brx}}}
\makeatother

\allowdisplaybreaks[1]

\makeatletter
\@namedef{subjclassname@2020}{%
  \textup{2020} Mathematics Subject Classification}
\makeatother


\begin{document}


\title[Relative simplicity and Tsuboi's metric]{Relative simplicity of the universal coverings of transformation groups and Tsuboi's metric}

\author[M.~Kawasaki]{Morimichi Kawasaki}
\address[Morimichi Kawasaki]{Department of Mathematics, Faculty of Science, Hokkaido University, North 10, West 8, Kita-ku, Sapporo, Hokkaido 060-0810 Japan}
\email{kawasaki@math.sci.hokudai.ac.jp}

\author[M.~Kimura]{Mitsuaki Kimura}
\address[Mitsuaki Kimura]{Department of Mathematics, Osaka Dental University, 8-1 Kuzuha-hanazono-cho, Hirakata, Osaka 573-1121 Japan}
\email{mkimura@math.kyoto-u.ac.jp}

\author[H.~Kodama]{Hiroki Kodama}
\address[Hiroki Kodama]{
\begin{itemize}
    \item Center for Mathematical Engineering, Musashino University, 
    3-3-3 Ariake, Koto-ku, Tokyo 135-8181 Japan
    \item Interdisciplinary Theoretical and Mathematical Sciences Program, RIKEN, 2-1 Hirosawa, Wako-shi, Saitama 351-0198 Japan
\end{itemize}}
\email{kodamahiroki@gmail.com}

\author[Y.~Matsuda]{Yoshifumi Matsuda}
\address[Yoshifumi Matsuda]{Department of Mathematical Sciences, College of Science and Engineering, Aoyama Gakuin University, 5-10-1 Fuchinobe, Chuo-ku, Sagamihara, Kanagawa 252-5258 Japan}
\email{ymatsuda@math.aoyama.ac.jp}

\author[T.~Matsushita]{Takahiro Matsushita}
\address[Takahiro Matsushita]{Department of Mathematical Sciences, Faculty of Science, Shinshu University, Nagano 390-8621 Japan}
\email{matsushita@shinshu-u.ac.jp}

\author[R.~Orita]{Ryuma Orita} 
\address[Ryuma Orita]{Department of Mathematics, Faculty of Science, Niigata University, Niigata 950-2181 Japan}
\email{orita@math.sc.niigata-u.ac.jp}

\dedicatory{Dedicated to Professor Takashi Tsuboi on the occasion of his $70$th birthday}


\begin{abstract}
Many transformation groups on manifolds are simple, but their universal coverings are not. In the present paper, 
we study the concept of \emph{relatively simple group}, that is, a group with the maximum proper normal subgroup.
We show that many examples of universal coverings of transformation groups are relatively simple, 
including the universal covering $\widetilde{\mathrm{Ham}}(M,\omega)$ of the group of Hamiltonian diffeomorphisms of a closed symplectic manifold $(M,\omega)$. 

Tsuboi constructed a metric space $\mathcal{M}(G)$ for a simple group $G$. 
We generalize his construction to relatively simple groups, and study their large scale geometric structure. 
In particular, Tsuboi's metric space of $\widetilde{\mathrm{Ham}}(M, \omega)$ is not quasi-isometric to the half line for every closed symplectic manifold $(M,\omega)$.

\end{abstract}

\maketitle

\tableofcontents

\section{Introduction}
\subsection{Background}
Many transformation groups regarding manifolds are known to be simple groups. 
For example, the following groups are simple:
\begin{enumerate}
    \item The identity component $\mathrm{Diff}^c_0(M)$ of the group $\mathrm{Diff}^c(M)$ of diffeomorphisms on $M$ with compact supports.
    \item The kernel of the volume-flux homomorphism. This is a normal subgroup of the group of volume-preserving diffeomorphisms.
    \item The group of Hamiltonian diffeomorphisms $\Ham(M, \omega)$ for a closed symplectic manifold $(M, \omega)$.
\end{enumerate}

The simplicity of these groups is sometimes useful (for example, see \cite{EP93}, \cite{BK13}, \cite{BKS} and \cite{I18}).

We would like to consider an analogue of simplicity on the universal coverings of these groups since many interesting and important invariants are defined over the universal covering of transformation groups. 
As the first example, the symplectic (resp. volume) flux homomorphism is originally defined as a map from the universal covering of the identity component of the group of symplectomorphisms (resp. volume-preserving diffeomorphisms)\footnote{This map induces the map on the identity component of the group of symplectomorphisms and it is also called the symplectic flux homomorphism}.
As the second example, for an open symplectic manifold, we define the Calabi homomorphism as a real-valued map on the universal covering of the group of Hamiltonian diffeomorphisms (see Subsection \ref{symp basis}).
As the third example, after Entov--Polterovich's famous work \cite{EP03}, many researcher constructed partial Calabi quasi-morphisms (see Definition \ref{def_of_PCQ}) on the universal covering of the group of Hamiltonian diffeomorphisms (for instance, see \cite{PR}, \cite{FOOO} and \cite{En14}).
It is actually proved that some of them cannot be defined on the group of Hamiltonian diffeomorphisms (see \cite{O06}, \cite{FOOO}).
We note that partial Calabi quasi-morphisms are very useful to prove non-displaceability of a fiber of an integrable system (for instance see \cite{BEP}, \cite{EP09}, \cite{FOOO}, \cite{KO21}) and it is known that descendability of partial Calabi quasi-morphisms to the group of Hamiltonian diffeomorphisms is not necessary to obtain this application.

However, the universal covering of a transformation group is not simple in many cases.
This is because the fundamental group of the transformation group is a normal subgroup of the universal covering.

In the present paper, we use the following notion of relative simplicity in order to consider an analogue of simplicity on the universal covering of transformation groups:

\begin{definition} \label{definition relatively simple}
Let $G$ be a group. A proper normal subgroup $N$ is the \emph{maximum normal subgroup of $G$} if one of the following equivalent conditions is satisfied:
\begin{enumerate}[(1)]
\item Every proper normal subgroup $N_0$ of $G$ is contained in $N$.

\item For every $x \in G \setminus N$, there is no proper normal subgroup of $G$ containing $x$. In other words, every element of $G \setminus N$ is a normal generator of $G$.
\end{enumerate}

We say that a group $G$ is \emph{simple relative to $N$} if $N$ is the maximum normal subgroup. A group $G$ is \emph{relatively simple} if $G$ has the maximum normal subgroup.
\end{definition}

Clearly, the notion of relative simplicity is a generalization of simplicity of groups. Indeed, a group $G$ is simple if and only if $G$ is relatively simple to its trivial subgroup. 
We also note that the notion of relative simplicity was essentially considered by \cite{KLM} and \cite{FY22} (see Definition \ref{FY_def} and Theorems \ref{KLM_Lie}, \ref{FY_rel_simple}.

For other notions related to relative simplicity, see Appendix \ref{other_appen}.

Some of our results (Theorems \ref{auto_thm}, \ref{closed_lphofer} and \ref{open_lphofer}) may be folklore.
However, the notion of relative simplicity may be useful for understanding the proof of these results clearly.


\subsection{Relative simplicity of coverings of transformation groups}

Let $G$ be a group, and $N$ a normal subgroup of $G$. If $G$ is simple relative to $N$, then $G/N$ is a simple group, but the converse does not hold in general (consider $G = \ZZ \times (\ZZ / 2\ZZ)$ and $N = \ZZ \times \{ 0\}$). However, in the present paper, we provide a criterion to show the relative simplicity of $G$ from the simplicity of $G/N$ in the case that the projection $G \to G/N$ is a covering map between topological groups. Using this, we can show the relative simplicity by the fact that the quotient group is simple. To state the criterion precisely, we recall a few notions from general topology.

A topological space $X$ is \emph{Polish} if it is homeomorphic to a separable complete metric space. A \emph{Polish group} is a topological group whose underlying topological space is Polish. Then we prove the following theorem:

\begin{thm}\label{mtst_simple}
Let $p \colon G \to H$ be a continuous surjective homomorphism between topological groups, and $K$ the kernel of $p$. Assume that the following conditions hold:
\begin{enumerate}[$(1)$]
\item $H$ is simple.

\item There exists an increasing sequence of Polish subgroups
\[ H_0 \subset H_1 \subset H_2 \subset \cdots\]
of $H$ such that
\[ \bigcup_{n \in \NN} H_n = H.\]
The topology of $H_n$ is the subspace topology of $H$. Moreover, assume that for every $x \in G$, there exists $n \in \mathbb{N}$ such that $x$ is contained in the identity component $p^{-1}(H_n)_0$ of $p^{-1}(H_n)$.


\item $G$ is connected.

\item $K$ is a discrete subgroup of $G$.
\end{enumerate}
Then $G$ is simple relative to $K$.
\end{thm}

Note that Theorem \ref{mtst_simple} deduces the following fact (see \cite[Lemma 3.4]{KLM}). Recall that a simple Lie group is a connected non-commutative Lie group which has no non-trivial connected normal subgroup.

\begin{thm}\label{KLM_Lie}
Let $G$ be a simple Lie group and $Z(G)$ the center of $G$.
Then $G$ is simple relative to $Z(G)$.
\end{thm}

It is fairly easy to show that several examples considered in this paper satisfy the assumptions of Theorem \ref{mtst_simple}. For example, Theorem \ref{mtst_simple} deduces the following:

\begin{thm}\label{smooth_rel_simpl}
Let $M$ be a manifold.
Set $G = \widetilde{\mathrm{Diff}_0^c}(M)$ and $N = \pi_1\left(\mathrm{Diff}_0^c(M))\right)$.
Then, $G$ is simple relative to $N$.
\end{thm}

However, in the case of the groups of Hamiltonian diffeomorphisms of symplectic manifolds, it is a difficult problem to determine whether the transformation group satisfies the assumptions in Theorem \ref{mtst_simple}. In fact, it is a quite subtle problem whether the natural projection $\tHam(M,\omega) \to \Ham(M, \omega)$ is a covering space in the sense of algebraic topology \cite{Hatcher} or not (see Remarks \ref{rem_ono}, \ref{rem_kislev}).
However, we can also show the following:

\begin{thm}\label{closed_rel_simpl}
Let $(M,\omega)$ be a closed symplectic manifold. Set $G = \tHam(M,\omega)$ and $N = \pi_1\left(\Ham(M,\omega)\right)$. Then, $G$ is simple relative to $N$.
\end{thm}

We have a result similar to Theorem \ref{closed_rel_simpl} in the case of open symplectic manifolds. The definition of the Calabi homomorphism $\Cal_M \colon \tHam(M,\omega) \to \RR$ will be provided in Subsection \ref{symp basis}.

\begin{thm}\label{open_rel_simpl}
Let $(M,\omega)$ be an open symplectic manifold.
Let $G=\Ker(\Cal_M)$ and $N=\pi_1\left(\Ham(M,\omega)\right) \cap G$.
Then, $G$ is simple relative to $N$.
\end{thm}

We will explain what is $G/N$ in the setting of Theorem \ref{open_rel_simpl} (see Theorem \ref{banyaga_open} and Proposition \ref{open_rel}).

Note that it is easy to construct an example where $G$ is not relatively simple to $K$ if any of (1), (3), or (4) of Theorem \ref{mtst_simple} fails to hold. However, the authors are not aware of any example where $G$ is not relatively simple to $K$ when only (2) fails to hold. Therefore, we propose the following problem.

\begin{prob}
Let $p \colon G \to H$ be a continuous surjective homomorphism between topological groups, and $K$ the kernel of $p$. Assume the conditions (1), (3) and (4) in Theorem \ref{mtst_simple}. Then is it true that $G$ is simple relative to $K$?
\end{prob}

\cite{FY22} implicitly proved the following theorem (use Example 3.1[II], Lemma 3.2(1) and Proposition 4.1 of \cite{FY22}).


\begin{thm}\label{FY_rel_simple}
Let $\pi \colon M\to S^1$ be a $C^\infty$ fiber bundle whose fiber $N$ is a closed manifold and $G$ be the identity component of the group of $C^\infty$ bundle diffeomorphisms of $\pi$.
Let $P\colon G \to \mathrm{Diff}_0(S^1)$ denote the natural homomorphism from $G$ to $\mathrm{Diff}_0(S^1)$ induced by the bundle structure $\pi$ and set $N= \mathrm{Ker}(P)$.
Then $G$ is simple relative to $N$.
\end{thm}

We end this subsection by providing remarks on previous work. 
The uniform version of the notion of relative simplicity is already considered by Fukui and Yagasaki.

\begin{definition}[\cite{FY22}]\label{FY_def}
Let $G$ be a group and $N$ a normal subgroup of $G$.
The group $G$ is called \textit{uniformly simple relative to} $N$ if, there exists $k \in \ZZ_{>0}$ such that, for every $f \in G$ and $g \in G \setminus N$, $f$ is represented as a product of $k$ or fewer number of conjugates of $g$ or $g^{-1}$.
\end{definition}

Fukui and Yagasaki proved the uniform relative simplicity of $G$ in Theorem \ref{FY_rel_simple} under some condition strogner than Theorem \ref{FY_rel_simple} (see Corollary 1.1 of \cite{FY22}).

\subsection{Some applications of relative simplicity}\label{subsec:some appli}

In this subsection, we provide some applications of Theorems \ref{closed_rel_simpl}, \ref{open_rel_simpl}.
The applications in this subsection may be folklore because their proofs are similar to the original ones.
However, the notion of relative simplicity may be useful for understanding the proof of these results clearly.

\begin{thm}\label{auto_thm}
Let $(M,\omega)$ be a symplectic manifold.
Every $f \in \tHam(M,\omega)$ can be represented as a product of autonomous elements of $\tHam(M,\omega)$.
\end{thm}

 For the definition of autonomous element, see Section \ref{well-def_aut}.
For $f \in \tHam(M, \omega)$, define $\| f\|_{aut}$ to be the smallest number $k$ such that there are autonomous $k$ elements whose product is $f$. We call this function $\|\cdot\|_{aut}\colon \tHam(M,\omega) \to \ZZ_{\geq0}$ the \emph{autonomous norm of $(M, \omega)$}.

Theorem \ref{auto_thm} means that the autonomous norm is well-defined on $\tHam(M,\omega)$.
We have results similar to Theorem \ref{auto_thm} on other transformation groups (Theorems \ref{smooth_auto}, \ref{vol_auto}).

The autonomous norm on $\Ham(M,\omega)$ is defined similarly.
Brandenbursky, K\k{e}dra and Shelukhin studied the autonomous norms on $\Ham(M,\omega)$ in \cite{BK13}, \cite{BKS} when the dimension of $M$ is two.

The following results are analogues of results in \cite{EP93} (for a historical background, see Subsection \ref{sec:Hofer_metric}).

\begin{thm}\label{closed_lphofer}
Let $(M,\omega)$ be a closed symplectic manifold.
For every integer $1 \leq p < \infty$, the $L^p$-Hofer metric $($see Section \ref{sec:Hofer_metric} for the definition$)$ on $\tHam(M,\omega)$ vanishes i.e., for every $f,g \in \tHam(M,\omega)$, $\tilde{\rho}_p(f,g) = 0$.
\end{thm}

\begin{thm}\label{open_lphofer}
Let $(M,\omega)$ be an open symplectic manifold.
For every integer $1 \leq p < \infty$ and for all $f \in \tHam(M,\omega)$,
\[ \tilde{\rho}_p(\Id, f) = (\mathrm{Volume}(M))^{(1-p)/p} \cdot \lvert \Cal(f) \rvert \]
holds.
\end{thm}

In Theorem \ref{open_lphofer}, we consider that $\infty^0=1$ for $p=1$ and $\infty^{(1-p)/p}=0$ for $p>1$.


\subsection{Tsuboi metric space}\label{TKI_metric_intro}
In \cite{T09}, Tsuboi constructed a metric space for a simple group $G$ (see also \cite{K11}, \cite{T17} and \cite{I18}). In the present paper, we generalize his construction for relatively simple groups in the following way.

Let $K$ be a normally generating subset of a group $G$. For $g \in G$, let $q_K(g)$ be the smallest number $n$ such that $g$ is a product of conjugates of elements of $K \cup K^{-1}$. Here we write $K^{-1}$ to indicate $\{ g^{-1} \mid g \in K\}$. This function $q_K(g)$ has appeared in the literature in conjugation invariant norms (see Definition \ref{cin}). When $K$ consists of a single element $x$, then we write $q_{x}(g)$ instead of $q_{\{ x\}}(g)$.

Let $G$ be a relatively simple group whose maximum normal subgroup is $N$. Elements $f$ and $g$ of $G$ are \emph{symmetrized conjugate} to each other if $f$ is conjugate to $g$ or $g^{-1}$, and let us denote $f \sim g$.
It is not difficult to confirm that symmetrized conjugacy is an equivalence relation. 
Let $[g]$ denote the symmetrized conjugacy class represented by $g \in G$. 
We set $\TK(G) = (G \setminus N)/\sim$ and define a function $d\colon \TK(G) \times \TK(G) \to \RR_{\geq0}$ by
\[
d([f], [g]) = \log\max\{q_{g}(f), q_{f}(g)\}.
\]
This function $d$ is a metric on $\TK(G)$.
Indeed, the triangle inequality of $d$ follows from the inequality
\[
q_{f}(h) \leq q_{f}(g)q_{g}(h),
\]
which we can confirm easily. 
In present paper, we call this $d$ the \textit{Tsuboi metric of $G$} and the pair $(\TK(G),d)$ the \textit{Tsuboi metric space of $G$}.
We note that $d$ is bounded if and only if $G$ is uniformly simple relative to its maximum normal subgroup.

Tsuboi considered the above metric space $(\TK(G),d)$ for a simple group $G$ \cite{T09}, and the above construction is a straightforward generalization to the case of relatively simple groups.

The Tsuboi metric spaces of simple groups have been studied with a focus on their large scale geometric structure. For example, 
the third author
\cite{K11} determined the quasi-isometric structure of the Tsuboi metric space of $A_\infty$.

\begin{theorem}[{\cite{K11}}] \label{theorem A_infinity_intro}
Let $A_\infty$ be the infinite alternating group $\bigcup_{n=1}^{\infty}A_n$.
The Tsuboi metric space $(\TK(A_\infty), d)$ is quasi-isometric to the half line $\RR_{\ge 0}$ with the standard metric.
\end{theorem}

In \cite{T17}, Tsuboi asked whether $\TK(\Ker(\Cal_{D^2}))$ is quasi-isometric to the half line with the standard metric or not. Here, $\Cal\colon \Ham(D^2,\omega_0) \to \RR$ is the Calabi homomorphism on the $2$-dimensional disk $D^2$ with the standard symplectic form $\omega_0$. Ishida \cite{I18} solved Tsuboi's problem.

\begin{thm}[{\cite[Theorem1.3]{I18}}]
The Tsuboi metric space $(\TK(\Ker (\Cal_{D^2})), d)$ is not quasi-isometric to the half line $\RR_{\ge 0}$ with the standard metric.
\end{thm}

Moreover, Ishida proved the following theorem.


\begin{thm}[\cite{I18}, Remark 3.2]\label{original ishida2}
Let $(M,\omega)$ be a one point (symplectic) blow up of a closed symplectic 4-manifold $(X,\omega_X)$ such that $\omega_X$ and the first Chern class $c_1(X)$ vanish on $\pi_2(X)$,
Then, $\left(\mathcal{M}(\Ham(M,\omega)),d\right)$ is not quasi-isometric to the half line $\RR_{\ge 0}$ with the standard metric.
\end{thm}

In light of these developments, we also study the quasi-isometric structures of the Tsuboi metric spaces for relatively simple groups. We first consider the Tsuboi metric space of $\tHam(M, \omega)$, and proved the below theorem (Theorem \ref{general ishida}). In this case, we can remove several hypothesises in Theorem \ref{original ishida2}, and we only need to assume that $(M, \omega)$ is a closed symplectic manifold.

\begin{thm}\label{general ishida}
Let $(M,\omega)$ be a closed symplectic manifold.
Then, $\left(\mathcal{M}(\tHam(M,\omega)),d\right)$ is not quasi-isometric to the half line $\RR_{\ge 0}$ with the standard metric.
\end{thm}

We also have the relative version of 
Theorem \ref{theorem A_infinity_intro}.

\begin{theorem} \label{theorem S_infinity_intro}
Let $S_\infty$ be the infinite symmetric group $\bigcup_{n=1}^\infty S_n$.
The metric space $(\TK(S_\infty), d)$ is quasi-isometric to the half line $\RR_{\ge 0}$ with the standard metric.
\end{theorem}
Note that the infinite symmetric group $S_\infty$ is relatively simple (Lemma \ref{lemma S relatively simple}) and thus we can define the metric space $(\TK(S_\infty), d)$.

\subsection{Organization of the paper}
In section \ref{sec:prelim}, we review preliminaries, including symplectic geometry. 
In section \ref{ham_simp_sec}, we prove the relative simplicity of transformation groups which comes from symplectic geometry (Theorems \ref{closed_rel_simpl} and \ref{open_rel_simpl}).
In section \ref{lphofer}, we discuss applications of relative simplicity to the $L^p$-Hofer metric (Theorems \ref{closed_lphofer} and \ref{open_lphofer}).
In section \ref{Yoshiko_section}, we consider the relative simplicity of covering groups of simple topological groups and establish Theorem \ref{mtst_simple}.
In section \ref{relative_simple_other}, we prove relative simplicity for some transformation groups that are not derived from symplectic geometry.
In section \ref{well-def_aut}, 

we prove that the autonomous norm is well-defined 

(Theorem \ref{auto_thm}).
In section \ref{tsuboi_metric} and \ref{symmetric_group}, we study the Tsuboi metric on the universal covering of Hamiltonian diffeomorphism groups and the infinite symmetric group, respectively. 
Finally, in Appendix \ref{other_appen}, we discuss quasi-simplicity and maximal normal subgroups.

\subsection*{Notation and conventions}

Unless otherwise specified, a manifold is assumed to be a connected $C^\infty$-manifold. Topological groups are assumed to be Hausdorff.

%

\section*{Acknowledgment}
The first author is supported in part by JSPS KAKENHI Grant Number 21K13790. 
The second author is partially supported by JSPS KAKENHI Grant Number 24K16921.
The fifth author is partially supported by JSPS KAKENHI Grant Number 23K12975.
The sixth author is partially supported by JSPS KAKENHI Grant Number 21K13787.

The authors would like to express their gratitude to 
Tatsuhiko Yagasaki for pointing out an error in the proof of Theorem 1.4 in the earlier draft.
They also would like to thank
Kazuhiro Kawamura,
Kyo Nishiyama,
Hiraku Nozawa,
for answering our questions and fruitful discussions.

\section{Preliminaries} \label{sec:prelim}

\subsection{The universal covering of transformation group}\label{def of univ cov}
A transformation group we treat in the present paper is not necessarily locally path-connected and so does not admit the universal covering in the sense of classical algebraic topology.
Here, we provide the definition of the universal covering of a subgroup of the group of smooth diffeomorphisms.
These are not necessarily local homeomorphisms (for example, see Remark \ref{rem_kislev}).
However, in accordance with convention, these are also referred to as the universal covering.

Let $M$ be a smooth manifold and $G$ a subgroup of $\mathrm{Diff}^c(M)$, where $\mathrm{Diff}^c(M)$ denotes the group of smooth diffeomorphisms with compact support.
An isotopy $\gamma\colon [0,1] \to G$ is called a \textit{smooth isotopy} in $G$ if the map $[0,1]\times M \to M$, $(t,x) \mapsto \gamma(t)(x)$ is a smooth map  and there exists a compact subset $K$ of $M$ such that for every $t\in[0,1]$, the support of $\gamma(t)$ is contained in $K$.
Let $\gamma_1,\gamma_2$ be smooth isotopies in $G$.
Smooth isotopies $\gamma_1$ and $\gamma_2$ in $G$ 
 are  called \textit{smoothly homotopic} in $G$ if there exists a map $\Gamma \colon [0,1]\times[0,1] \to G$ such that
\begin{itemize}
\item the map $[0,1]\times[0,1]\times M \to M$, $(s,t,x) \mapsto \Gamma(s,t)(x)$ is a smooth map.
\item $s\mapsto\Gamma(s,0)$, $s\mapsto\Gamma(s,1)$ are constant maps.
    \item $\Gamma(0,t)=\gamma_1(t)$, $\Gamma(1,t)=\gamma_2(t)$ for every $t\in[0,1]$.
\end{itemize}
For a smooth isotopy $\gamma$, let $[\gamma]$ denote the smooth homotopy class in $G$ relative to fixed ends.
Then, we define the set $\tilde{G}$ by
\[\tilde{G}=\{[\gamma] \mid \gamma\colon [0,1] \to G \text{ is a smooth isotopy with }\gamma(0)=\mathrm{Id}\}.\]

One can see that this $\tilde{G}$ is a group.
We define the map $\pi\colon \tilde{G} \to G$ to be the natural map attaining the end of the isotopy.
One can also see that this $\pi$ is a group homomorphism.
We call the group $\tilde{G}$ and the homomorphism $\pi\colon\tilde{G} \to G$ the \textit{universal covering of $G$}.
The kernel of the universal covering $\pi\colon\tilde{G} \to G$ is called the \textit{fundamental group of $G$} and denoted by $\pi_1(G)$.
 The image of the universal covering $\pi\colon\tilde{G} \to G$ is called the \textit{identity component of $G$}.
\begin{remark}
We note that $G$ is not necessarily locally  smoothly path-connected (see Remark \ref{rem_kislev}) and so the natural map $\pi\colon\tilde{G} \to G$ cannot be  genuinely a covering in the usual sense in algebraic topology.
However, due to the convention in the studies of transformation groups and symplectic geometry (for example, see \cite{Ban}, \cite{Ban97}, \cite{EP03} and \cite{O05}), we use the word ``universal covering''.
\end{remark}

%
%

\subsection{Relative simplicity}\label{base of maximum normal subgroup}

The purpose of this subsection is to summarize direct consequences that can be immediately derived from the definition of relative simplicity (see Definition \ref{definition relatively simple}), as well as related concepts and historical background.

Recall that a group $G$ is relatively simple if it has the maximum (proper) 
 normal
subgroup, and that if $N$ is the maximum 
 normal
subgroup of $G$, then $G$ is said to be simple relative to $N$.

\begin{example} \label{example relatively simple}
The following are examples of relatively simple groups:
\begin{enumerate}[(1)]
\item Every simple group $G$ is relatively simple since $G$ is simple relative to the trivial subgroup $\{ 1_G\}$ of $G$.

\item For a prime number $p$ and a positive integer $n$, $\ZZ / p^n \ZZ$ is simple relative to $p\ZZ / p^n \ZZ$.

\item For $n \ge 2$, the symmetric group $S_n$ is simple relative to the alternatig group $A_n$ (see Lemma \ref{lemma S relatively simple}).

\item Let $G$ be a relatively simple group, $H$ a non-trivial group, and $f \colon G \to H$ a surjective group homomorphism. Then $H$ is relatively simple.
\end{enumerate}
\end{example}

\begin{example} \label{example not relatively simple}
On the other hand, the following are examples of groups which are not relatively simple:
\begin{enumerate}[(1)]
\item The additive group $\ZZ$ of integers is not relatively simple. Indeed, there is no proper normal subgroup of $\ZZ$ containing both $2\ZZ$ and $3\ZZ$.

\item For distinct primes $p$ and $q$, $\ZZ / pq \ZZ$ is not relatively simple.

\item If $H \ne 1$ is not a relatively simple group and $f \colon G \to H$ is a surjective group homomorphism, then $G$ is not relatively simple. This is the contrapositive of (4) of Example \ref{example relatively simple}. For instance, a group admits a group homomorphism onto $\ZZ$ is not relatively simple.
\end{enumerate}
\end{example}



The following rephrasing of relative simplicity will be used in the subsequent sections without further notice.

\begin{lem} \label{rel simpl equiv}
Let $G$ be a group and $N$ a proper normal subgroup of $G$.
The group $G$ is simple relative to $N$ if and only if, for every $f \in G$ and $g \in G \setminus N$, $f$ is represented as a product of conjugates of $g$ or $g^{-1}$.
\end{lem}

We also have the following observation. The proof is immediate and hence is omitted.

\begin{prop}\label{rel_simple_is_simple}
Let $G$ be a group and $N$ a normal subgroup of $G$.
If $G$ is simple relative to $N$, then $G/N$ is a simple group.
\end{prop}

Note that the converse of Proposition \ref{rel_simple_is_simple} does not hold. Indeed, $\ZZ / 2\ZZ$ is simple, but $\ZZ$ is not relatively simple.   

\subsection{Symplectic geometry}\label{symp basis}
Let $(M,\omega)$ be a $2n$-dimensional symplectic manifold.

For a smooth function $F$ on $M$ with compact support, the \textit{Hamiltonian vector field} $X_{F}\in \mathfrak{X}(M)$ associated to $F$ is defined by
\[
	\iota_{X_{F}}\omega=-dF.
\]
Let $H$ be a time-dependent Hamiltonian with compact support,
i.e., a smooth function $H\colon [0,1] \times M\to\RR$ with compact support.
We set $H_t=H(t,\cdot) \colon M \to \RR$ for $t\in [0,1]$.
The \textit{Hamiltonian isotopy} $\{\varphi_H^t\}_{t\in\RR}$ associated to $H$ is defined by
\[
	\begin{cases}
		\varphi_H^0=\mathrm{Id},\\
		\frac{d}{dt}\varphi_H^t=X_{H_t}\circ\varphi_H^t\quad \text{for all}\ t\in\RR,
	\end{cases}
\]
and its time-one map $\varphi_H=\varphi_H^1$ is referred to as the \textit{Hamiltonian diffeomorphism with compact support} generated by $H$.
Let  $\Ham(M,\omega)$  denote the group of Hamiltonian diffeomorphisms of  $(M,\omega)$  with compact supports.

For an open symplectic manifold $(M,\omega)$, we recall that the \textit{Calabi homomorphism}
is a function $\mathrm{Cal}_{M} \colon \tHam(M,\omega)\to\mathbb{R}$ defined by
\[
	\mathrm{Cal}_{M}(\varphi_H)=\int_0^1\int_M H_t \omega^n\,dt,
\]
 where $H \colon [0,1] \times M \to \RR$ is a smooth function. It is known that the Calabi homomorphism is a well-defined group homomorphism (see \cite{Cala}, \cite{Ban}, \cite[Theorem 4.2.7]{Ban97}, \cite{MS}, \cite{PR}, \cite[Proposition 2.5.3]{Oh15a}).

The Calabi homomorphism induces a map on $\Ham(M,\omega)$ in the following way.
Set
$\Lambda_\omega = \Cal_M\left(\pi_1(\Ham(M,\omega))\right)$.
Then, we obtain the homomorphism
$\underline{\Cal}_M \colon \Ham(M,\omega) \to \RR/\Lambda_\omega$.
When $(M,\omega)$ is exact, i.e., $\omega$ is an exact form, it is known that $\Lambda_\omega=\{0\}$ and $\underline{\Cal}_M$ 
is a real-valued function
on $\Ham(M,\omega)$ and also called the Calabi homomorphism (see \cite{Cala}, \cite{Ban}, \cite[Proposition 4.2.12]{Ban97}, \cite{MS}, \cite{Hum}, \cite{PR}, \cite[Theorem 2.5.6]{Oh15a}).

Let $(M,\omega)$ be a symplectic manifold and $U$ a non-empty open subset of $M$.
Then, the inclusion $U \hookrightarrow M$ induces the homomorphism $p_{U,M} \colon \tHam(U,\omega|_U) \to \tHam(M,\omega)$, and we let $\tHam_U(M,\omega)$ denote the image of $p_{U,M}$.
An element $h$ of $\tHam(M,\omega)$ \textit{displaces} $U$ if $\pi(h)(U)\cap  \overline{U} =\emptyset$ where $\pi \colon \tHam(M,\omega) \to \Ham(M,\omega)$ is the universal covering.


Here, we review a partial Calabi quasi-morphism constructed by Entov and Polterovich \cite{EP06}. First, we review the definition of fragmentation norm.
Let $(M,\omega)$ be a symplectic manifold and $U$ a non-empty open subset of $M$.
Then, as we will explain in Lemma \ref{frag_lemma}, every $f \in \tHam(M,\omega)$ can be written as a product of conjugation of element of $\tHam_U(M,\omega)$.
We define the fragmentation norm $\|f\|_U$ of $f \in \tHam(M,\omega)$ 
as the minimum of such numbers. In other words,
 \[
\|f\|_U=\min
\left\{k\, \middle| \,
\begin{gathered} 
    \exists f_1,\exists f_2,\ldots,\exists f_k \in \tHam_U(M,\omega), \exists h_1,\exists h_2\ldots,\exists h_k \in \tHam(M,\omega) \\
    \text{ such that } f = h_1^{-1}f_1h_1 \cdot h_2^{-1}f_2h_2 \cdot \cdots \cdot h_k^{-1}f_kh_k
\end{gathered}
\right\}.
\] 


Next, we review the definition of partial Calabi quasi-morphism
\begin{definition}[\cite{EP06}, \cite{En14}]\label{def_of_PCQ}
    Let $(M,\omega)$ be a connected closed symplectic manifold.
    A function $\mu\colon \tHam(M,\omega)\to \RR$ is called a \textit{partial Calabi quasi-morphism} if $\mu$ satisfies the following conditions.

\begin{itemize}
  \item[$(1)$] (Partial homogeneity) $\mu(\phi^k) = k\mu(\phi)$ for every $\phi \in \tHam(M,\omega)$ and every $k \in \ZZ_{\geq 0}$.
  \item[$(2)$] (Partial quasi-additivity) For a displaceable open set $U \subset M$, there exists $C > 0$ so that 
  \[
  |\mu(\phi\psi) - \mu(\phi) - \mu(\psi)| \leq C \min\{\|\phi\|_U, \|\psi\|_U\}
  \]
  for every $\phi, \psi \in \tHam(M,\omega)$.
  \item[$(3)$] (Calabi property)
   Let $U$ be an open and displaceable subset of $M$, then $\mu\circ p_{U,m}(\phi)=\Cal_U(\phi)$ for every $\phi \in \tHam_U(M,\omega)$.
\end{itemize}
\end{definition}

\begin{remark}
There are different definitions of partial Calabi quasi-morphisms as in \cite{MVZ}, \cite{En14}, and \cite{KO19}.
\end{remark}

\begin{remark} \label{calabi qm remark}
When Entov and Polterovich first constructed partial Calabi quasi-morphisms, 
they imposed certain
assumptions on symplectic manifolds.
However, as is written in \cite[Subsection 22.4]{Oh15b} for instance, their construction works well for every closed symplectic manifold.
\end{remark}


\subsection{Hofer's metric and conjugation-invariant norms}\label{sec:Hofer_metric}

Let $p\geq 1$ be an integer.
Given a smooth function $H\colon M\to\RR$ with compact support, we define the \textit{$L^p$-norm} of $H$ to be
\[
    \|H\|_{L^p} = \left(\int_M \lvert H\rvert^p\omega^n\right)^{1/p}.
\]
Moreover we set
\[
    \|H\|_{L^{\infty}} = \sup_M{\lvert H\rvert}.
\]
For $p \geq 1$ or $p=\infty$ we define a function $\rho_p\colon \Ham(M,\omega)\times \Ham(M,\omega)\to \RR$ by the formula
\[
    \rho_p(\phi,\psi) = \inf_{\varphi_H^1=\phi^{-1}\psi}\int_0^1\|H_t\|_{L^p}\,dt,\quad \phi,\psi\in\Ham(M,\omega),
\]
where the infimum is taken over all Hamiltonian functions $H\colon [0,1]\times M\to\RR$ with compact supports
which define the Hamiltonian diffeomorphism $\phi^{-1}\psi$, and $H_t=H(t,\cdot)$ for each $t$.
Then the function $\rho_p$ defines a bi-invariant pseudo-metric on $\Ham(M,\omega)$ (see, e.g., \cite[Section 2.2]{P01}, \cite[Proposition 12.3.1]{MS}, \cite[Subsection 1.3.3]{PR}).
Moreover, for $p=\infty$, it is a deep fact that the pseudo-metric $\rho_{\infty}$ is a genuine metric on $\Ham(M,\omega)$ called the \textit{Hofer metric}  or \textit{Hofer's metric} \cite{Ho90,P93,LM95}.
On the other hand, Eliashberg and Polterovich \cite[Corollary 1.3.B]{EP93} showed that the pseudo-metric $\rho_p$ is not a metric for every positive integer $p$.

Similarly, one can define a bi-invariant pseudo-metric $\tilde{\rho}_p$ on the universal covering $\tHam(M,\omega)$ by the formula
\[
    \tilde{\rho}_p(\phi,\psi) = \inf_{[\{\varphi_H^t\}]=\phi^{-1}\psi}\int_0^1\|H_t\|_{L^p}\,dt,\quad \phi,\psi\in\tHam(M,\omega),
\]
where the infimum is taken over all Hamiltonian functions $H\colon [0,1]\times M\to\RR$ with compact supports
which define the Hamiltonian isotopy $\phi^{-1}\psi$.

We define a function $\tilde{\nu}_p \colon \tHam(M,\omega) \to\RR$ by
\[\tilde{\nu}_p(f) = \tilde{\rho}_p({\rm Id},f).\]
This $\tilde{\nu}_p$ is a conjugation-invariant pseudo-norm in the sense of \cite{BIP}.
Here, we review the definition of the conjugation-invariant norm.
\begin{definition}[{\cite{BIP}}]\label{cin}
Let $G$ be a group. A function $\nu\colon G\to \mathbb{R}_{\geq0}$ is \textit{ a conjugation-invariant norm} on $G$ if $\nu$ satisfies the following axioms:
\begin{itemize}
\item[(1)] $\nu(1)=0$;
\item[(2)] $\nu(f)=\nu(f^{-1})$ for every $f\in G$;
\item[(3)] $\nu(fg)\leq \nu(f)+\nu(g)$ for every $f,g\in G$;
\item[(4)] $\nu(f)=\nu(gfg^{-1})$ for every $f,g\in G$;
\item[(5)] $\nu(f)>0$ for every $f\neq 1\in G$.
\end{itemize}
A function $\nu\colon G\to \mathbb{R}$ is a \emph{conjugation-invariant pseudo-norm} on $G$ if $\nu$ satisfies the above axioms (1),(2),(3) and (4).
\end{definition}

One of the most famous examples of conjugation-invariant norm is the commutator length.
The commutator length and its stabilization have a long history of study  (see Calegari's famous book \cite{Ca}).
For instance, the commuator length of the mapping class group is related to fiber structures of 4-dimensional manifolds and has studied for a long time  (see \cite{EK}, \cite{CMS}, and \cite{BBF}).
The commuator length of the transformation groups has been studied by various researchers (see \cite{BIP}, \cite{T09}, \cite{T17} and \cite{BHW}).

The fragmentation norm $\|\cdot\|_U\colon \tHam(M,\omega) \to \RR$ defined in Subsection \ref{symp basis} is also a conjugation-invariant norm.
The ``support'' of a group element is often regarded 
 as
a conjugation-invariant norm (Examples 1.5 and 1.19 of \cite{BIP}).
In this paper, we similarly regard ``support" as a conjugation-invariant norm (see equation (\ref{symmetric support})).



\section{Relative simplicity in symplectic geometry}\label{ham_simp_sec}

In this section, we will prove Theorems \ref{closed_rel_simpl} and \ref{open_rel_simpl}.
We  note that the proofs of Theorems \ref{closed_rel_simpl} and \ref{open_rel_simpl} are parallel to those  of simplicity of $\Ham(M,\omega)$ and $\Ker(\Cal)$ \cite{Ban}.

We use the following lemma.
\begin{lem}\label{disp_lem}
Let $(M,\omega)$ be a symplectic manifold and $U$ a non-empty open subset of $M$.
Assume that $h \in \tHam(M,\omega)$ displaces $U$.
Then for every $f,g \in \tHam_U(M,\omega)$
the commutator $[f,g]$ can be written as a product of 
 symmetrized
conjugates of $h$.
\end{lem}

An argument similar to Lemma \ref{disp_lem} in the case of $\Diff(M)$ or $\Ham(M,\omega)$ is very important and well known in the group theoretic study of transformation groups (for example, see \cite{Ban}, \cite[Lemma 2.3.B]{EP93}, \cite[Theorem 2.1.7]{Ban97}, \cite[Lemma 2.4.B]{P01}, and \cite[Theorem 2.2 (ii)]{BIP}).
Lemma \ref{disp_lem} is a variant of that classical argument to the universal covering.

\begin{proof}[Proof of Lemma \ref{disp_lem}]
Since $h$ displaces $U$, $hg^{-1}h^{-1} \in \tHam_{h(U)}(M,\omega)$ commutes with $f,g \in \tHam_U(M,\omega)$.
Hence,
\begin{align*}
[f,g] & = fgf^{-1}g^{-1} \\
& =  fg(hg^{-1}h^{-1})f^{-1}(hg^{-1}h^{-1})^{-1}g^{-1} \\
& =  \left((fg)h(fg)^{-1}\right)(fh^{-1}f^{-1})h(gh^{-1}g^{-1}).
\end{align*}
Thus, $[f,g]$ can be written as a product of symmetric conjugates of $h$.
\end{proof}

First we prove Theorem \ref{closed_rel_simpl}.
We use the following lemma.
\begin{lem}[Banyaga's fragmentation lemma]\label{frag_lemma}
Let $(M,\omega)$ be a symplectic manifold and $U$ a non-empty open subset of $M$.
Then, every $f \in \tHam(M,\omega)$ can be written as a product of conjugates of elements of $\tHam_U(M,\omega)$.
Moreover, if $M$ is closed, every $f \in \tHam(M,\omega)$ can be written as a product of conjugates of elements of $\Ker(\Cal_U)$.
\end{lem}
We note that Banyaga proved a statement similar to Lemma \ref{frag_lemma} on $\Ham(M,\omega)$ (\cite[Lemme de fragmentation (III.3.2)]{Ban}, see also {\cite[Subsection 4.2]{Ban97}}).
His proof works well for Lemma \ref{frag_lemma} similarly because he considered fragmentation for an isotopy in the proof and so we omit the proof of Lemma \ref{frag_lemma}.

%

We use the following theorem.
\begin{thm}\label{symp_perfect}
Let $(\mathbb{B}^{2n}(r),\omega_0)$ be the ball $\{\,x \in \RR^{2n} \mid \|x\|<r\,\}$ of radius $r$ associated with the standard symplectic form $\omega_0$ and $\Cal_B$ the Calabi homomorphism on $\tHam(\mathbb{B}^{2n}(r),\omega_0)$.
Then, $\Ker(\Cal_B)$ is perfect.
\end{thm}

The authors know no references explaining Theorem \ref{symp_perfect} explicitly, so we provide its proof.
In order to prove Theorem \ref{symp_perfect}, we use the following lemma.

\begin{lem}\label{univ_ker_cal_lem}
Let $(M,\omega)$ be an open symplectic manifold and $\Cal_M$ be the Calabi homomorphism on $\tHam(M,\omega)$.
Let $\tilde{\iota} \colon \tKer(\bCal_M) \to \tHam(M,\omega)$ be the homomorphism induced from the inclusion $\iota \colon \Ker(\bCal_M) \to \Ham(M,\omega)$, where $\tKer(\bCal_M)$ is the universal covering of $\Ker(\bCal_M)$.
Then, $\mathrm{Im}\tilde\iota=\Ker(\Cal_M)$.
\end{lem}

\begin{proof}
Recall that we took the identity as a base point when we consider the universal covering of transformation group.
Thus, we see that $\mathrm{Im}\tilde\iota \subset \Ker(\Cal_M)$.
We will prove $\Ker(\Cal_M)\subset\mathrm{Im}\tilde\iota$.
Take $h \in \Ker(\Cal_M)$.
Then, there exists a Hamiltonian function $H\colon [0,1] \times M \to \RR$ with compact support such that the Hamiltonian isotopy $\{\varphi_H^t\}_{t\in[0,1]}$ represents $h$.
Take a function $F\colon M \to \RR$ such that $\Cal_M(\left[\{\varphi_F^t\}_{t\in[0,1]}\right])=1$.
Then, we define
the Hamiltonian isotopy $\{\psi^t\}_{t\in[0,1]}$ by
\[\psi^t=\varphi_H^t \circ \varphi_F^{-\int_0^t(\int_MH_\tau\omega^n)d\tau} \text{ for every }t\in[0,1].\]
Then $\{\psi^t\}_{t\in[0,1]}$
 represents $h$ in $\tHam(M,\omega)$ because
 \[\Phi(s)=\left\{\varphi_H^t \circ \varphi_F^{-s\int_0^t(\int_MH_\tau\omega^n)d\tau}\right\}_{t\in[0,1]} (s \in [0,1])
 \] is a homotopy between $\{\varphi_H^t\}_{t\in[0,1]}$ and $\{\psi^t\}_{t\in[0,1]}$.

By $\Cal_M(\left[\{\varphi_F^t\}_{t\in[0,1]}\right])=1$, we have $\bCal_M \left(\varphi_F^a\right) = \bCal_M \left(\varphi_{aF}^1\right) =[a] \in \RR/\Lambda_\omega$ for every $a \in \RR$.
Hence, we have
\begin{align*}
\bCal_M\left(\psi^t\right) & =
\bCal_M\left(\varphi_H^t \circ \varphi_F^{-\int_0^t (\int_MH_\tau\omega^n)d\tau}\right) \\ 
& = \bCal_M\left(\varphi^t_H \right)+\bCal_M \left(\varphi_F^{-\int_0^t (\int_MH_\tau\omega^n)d\tau}\right) \\
& =  \left[\int_0^t\left(\int_MH_\tau\omega^n\right)d\tau-\int_0^t\left(\int_MH_\tau\omega^n\right)d\tau\right]\\
& =[0] \in \RR/\Lambda_\omega
\end{align*}
for every $t\in[0,1]$.
Therefore, the isotopy $\{\psi^t\}_{t\in[0,1]}$ represents a path in $\Ker(\bCal_M)$.
Hence, we have $h \in \mathrm{Im}\tilde\iota$.
\end{proof}


\begin{proof}[Proof of Theorem \ref{symp_perfect}]
In the proof of \cite[Th\'{e}or\`{e}me II.6.1]{Ban} (see also \cite[Subsection 4.5]{Ban97}), Banyaga proved that $\tKer(\bCal_B)$ is perfect.
Hence, by
Lemma \ref{univ_ker_cal_lem}, $\Ker(\Cal_B)$ is perfect.
\end{proof}

\begin{proof}[Proof of Theorem \ref{closed_rel_simpl}]
Take $f \in G$ and $g \in G \setminus N$.
By Lemma \ref{rel simpl equiv}, it is sufficient to prove that $f$ is represented as a product of conjugates of $g$ or $g^{-1}$.

Since $g\notin\pi_1(\Ham(M,\omega))$, there exists an open subset $U$ of $M$ such that $g$ displaces $U$ and $(U,\omega|_U)$ is symplectomorphic to a ball with the standard symplectic form.
Then, by Lemma \ref{frag_lemma}, $f$ can be written as a product of conjugates of elements of $p_{U,M}\left(\Ker(\Cal_U)\right)$.
By Theorem \ref{symp_perfect}, $p_{U,M}\left(\Ker(\Cal_U)\right)$ is a perfect group.
Thus,
since $g$ displaces $U$, Lemma \ref{disp_lem} implies that $f$ is represented as a product of conjugates of $g$ or $g^{-1}$.
\end{proof}

In order to prove Theorem \ref{open_rel_simpl}, we use the following lemma.

\begin{lem}[Fragmentation lemma for Theorem \ref{open_rel_simpl}.]\label{frag_lemma_open}
Let $(M,\omega)$ be a connected open symplectic manifold and $U$ a non-empty open subset of $M$.
Then, for every $f \in \Ker(\Cal_M)$, there exist $f_1,f_2\ldots,f_k \in p_{U,M}\left(\Ker(\Cal_U)\right)$ and $h_1,h_2\ldots,h_k \in \Ker(\Cal_M)$ such that
\[f = h_1^{-1}f_1h_1 \cdot h_2^{-1}f_2h_2 \cdots h_k^{-1}f_kh_k.\]
%
%
\end{lem}

We omit the proof of Lemma \ref{frag_lemma_open} because its proof is similar to {\cite[Subsection 4.2]{Ban97}}.

The proof of Theorem \ref{open_rel_simpl} is almost same as the one of Theorem \ref{closed_rel_simpl} if one use Lemma \ref{frag_lemma_open} instead of Lemma \ref{frag_lemma}.
Thus we omit the proof.


\section{$L^p$-Hofer geometry on the universal covering}\label{lphofer}

In this section, we prove Theorems \ref{closed_lphofer} and \ref{open_lphofer}.
 Recall that the conjugation-invariant norm $\tilde{\nu}_p \colon \tHam(M,\omega) \to\RR$ is defined by
\[\tilde{\nu}_p(f) = \tilde{\rho}_p(\Id,f).\]
Since $\tilde{\nu}_p$ is a conjugation-invariant norm, the set
\[
    N_{\tilde{\nu}_p} = \left\{\, f\in \tHam(M,\omega) \relmiddle| \tilde{\nu}_p(f)=0 \,\right\}
\]
is a normal subgroup of $\tHam(M,\omega)$.

Using Eliashberg--Polterovich's beautiful argument in \cite[Corollary 1.3.B]{EP93}
(see also \cite[Theorem 2.3.A]{P01}),
we have the following lemma.

\begin{lem}\label{EP_lemma}
Let $(M,\omega)$ be a symplectic manifold.
Let $p\geq 1$ be an integer.
Then
\[
    N_{\tilde{\nu}_p} \setminus \pi_1(\Ham(M,\omega)) \neq\emptyset.
\]
Moreover, if $M$ is open, then
\[
    N_{\tilde{\nu}_p}\cap\bigl(\Ker(\Cal_M)\setminus \pi_1(\Ham(M,\omega))\bigr)\neq\emptyset.
\]
\end{lem}

We note that Eliashberg and Polterovich proved a statement similar to Lemma \ref{EP_lemma} on $\Ham(M,\omega) \setminus \{\Id\}$.
However, their proof works well for Lemma \ref{EP_lemma} similarly and so we omit the proof of Lemma \ref{EP_lemma}.
Now combining Lemma \ref{EP_lemma} with Theorem \ref{closed_rel_simpl} yields Theorem \ref{closed_lphofer}.

\begin{proof}[Proof of Theorem \ref{closed_lphofer}]
By Theorem \ref{closed_rel_simpl}, $\tHam(M,\omega)$ is simple relative to $\pi_1(\Ham(M,\omega))$.
Since $N_{\tilde{\nu}_p}$ is a normal subgroup of $\tHam(M,\omega)$, Lemma \ref{EP_lemma} yields that
$N_{\tilde{\nu}_p} = \tHam(M,\omega)$.
It means that every $f\in\tHam(M,\omega)$ satisfies $\tilde{\nu}_p(f)=0$.
\end{proof}

To show Theorem \ref{open_lphofer}, we define the notion of \textit{relative Calabian} groups,
which generalizes Calabian groups introduced by Eliashberg and Polterovich \cite[Section 3.1]{EP93}.

\begin{definition}
    Let $G$ be a topological group and $N$ a normal subgroup of $G$.
    The group $G$ is called \textit{Calabian relative to $N$} if it satisfies the following conditions:
    \begin{itemize}
        \item[(C1)] There exists a continuous surjective homomorphism $c\colon G\to\RR$.
        \item[(C2)] Every normal subgroup $S$ of $G$ satisfying $S\cap (\Ker{(c)}\setminus N)\neq \emptyset$ contains the kernel of $c$.
        \item[(C3)] There exists a continuous path $\gamma\colon\RR\to G$ such that $c\bigl(\gamma(t)\bigr)=t$ for every $t\in\RR$.
    \end{itemize}
\end{definition}

We note that a topological group $G$ is Calabian in the sense of Eliashberg--Polterovich if and only if it is Calabian relative to the trivial subgroup $\{1\}$.

\begin{example}\label{example:Relative_Calabian}
Let $(M,\omega)$ be an open symplectic manifold.
We equip the group $\Ham(M,\omega)$ of Hamiltonian diffeomorphisms with the topology induced by the Hofer metric $\rho_{\infty}$ (recall Section \ref{sec:Hofer_metric}).
Then the universal covering (in the sense of Section \ref{def of univ cov}) $\tHam(M,\omega)$ is Calabian relative to $\pi_1(\Ham(M,\omega))$ with continuous surjective homomorphism $\Cal\colon\tHam(M,\omega)\to\RR$.
Indeed, let $S$ be a normal subgroup of $\tHam(M,\omega)$ such that
\[
    S\cap \bigl(\Ker(\Cal_M)\setminus \pi_1(\Ham(M,\omega))\bigr)\neq \emptyset.
\]
Since $\Ker(\Cal_M)$ is simple relative to $\Ker(\Cal_M)\cap \pi_1(\Ham(M,\omega))$ by Theorem \ref{open_rel_simpl},
the normal subgroup $S\cap\Ker(\Cal_M)$ coincides with $\Ker(\Cal_M)$.
It means that $S\supset\Ker(\Cal_M)$.
\end{example}

Let $G$ be a topological group and $\rho$ a pseudo-metric on $G$.
We assume that $\rho$ is \textit{intrinsic} in the sense of Eliashberg--Polterovich \cite[Section 3.2]{EP93}.
Roughly speaking, this assumption says that the distance $\rho(p,q)$ of given two points $p,q$ in $G$ is measured by the infimum of all \textit{lengths} of (rectifiable) paths joining $p$ and $q$.
We define a pseudo-norm $\nu$ on $G$ by $\nu(g)=\rho(1,g)$ for every $g\in G$.
If the pseudo-metric $\rho$ is bi-invariant, then the pseudo-norm $\nu$ is conjugation-invariant and hence the set
\[
    N_\nu = \{\, g\in G \mid\, \nu(g)=0 \}
\]
is a normal subgroup of $G$.
The following theorem is an analogue of \cite[Theorem 3.3.A]{EP93} in the context of relative Calabian groups.

\begin{theorem}\label{thm:EP_Theorem_3.3.A}
    Let $G$ be a topological group which is Calabian relative to a normal subgroup $N$ of $G$.
    Let $\rho$ be a continuous, bi-invariant, and intrinsic pseudo-metric on $G$.
    Assume that $N_{\nu}\cap (\Ker{(c)}\setminus N)\neq\emptyset$.
    Then there exists a positive constant $\lambda$ such that
    \[
        \rho(1,g) = \lambda\lvert c(g)\rvert
    \]
    for every $g\in G$.
\end{theorem}

\begin{proof}
By assumption, the condition (C2) yields that $N_{\nu}\supset\Ker{(c)}$.
Now the exactly same argument in \cite[Lemmas 3.3.B--3.3.E]{EP93} completes the proof.
\end{proof}

Now we are ready to prove Theorem \ref{open_lphofer}.

\begin{proof}[Proof of Theorem \ref{open_lphofer}]
Let $p\geq 1$ be an integer.
Lemma \ref{EP_lemma} and Example \ref{example:Relative_Calabian} imply that one can apply Theorem \ref{thm:EP_Theorem_3.3.A} for $G=\tHam(M,\omega)$, $N=\pi_1(\Ham(M,\omega))$, $\rho=\tilde{\rho}_p$, $\nu=\tilde{\nu}_p$, and $c=\Cal$.
Then there exists a positive constant $\lambda_p$ such that $\tilde{\rho}_p(\Id,f)=\lambda_p\lvert\Cal(f)\rvert$ for every $f\in\tHam(M,\omega)$.
Now an argument similar to the proof of \cite[Corollary 1.4.B]{EP93} shows that the constant $\lambda_p$ coincides with $(\mathrm{Volume}(M,\omega))^{(1-p)/p}$.
\end{proof}

\section{Covering groups of simple topological groups}\label{Yoshiko_section}

The goal of this section is to prove Theorem \ref{mtst_simple}. In Subsection \ref{prel_of_Polish}, we review several definitions and facts concerning Polish groups. In Subsection \ref{subsection separable groups}, we discuss a basic property of the separability of topological groups. The results given in these subsections seem to be folklore, and are not new. We complete the proof of Theorem \ref{mtst_simple} in Subsection \ref{subsection proof of Polish groups}.


\subsection{Polish groups}\label{prel_of_Polish}

In this subsection, we review several facts of Polish groups which we will need later. First, we recall the definition of Polish group. For a comprehensive introduction to this subject, we refer the reader to \cite{Bourbaki} and \cite{Kechris}.

\begin{definition}
A \emph{Polish space} is a topological space that is homeomorphic to a separable complete metric space. A topological group $G$ is called a \emph{Polish group} if its underlying topological space is Polish.
\end{definition}

Here we review some known facts related to Polish spaces. In the proof of Theorem \ref{mtst_simple}, we will frequently use the following characterization of Polish subspaces:

\begin{lemma}[{\cite[Chapter 9 \S 6 $\mathrm{N}^\circ$ Th\'{e}or\`{e}me 1]{Bourbaki}}] \label{lemma Polish facts}
Let $X$ be a Polish space. A subset $A$ of $X$ is a Polish subspace if and only if $A$ is a $G_\delta$-set of $X$, i.e., a countable intersection of open subsets in $X$.
\end{lemma}

In particular, every open subset of a Polish space is a Polish subspace. A finite union of Polish subspaces is again a Polish subspace.

\begin{definition}
Let $X$ be a topological space.
\begin{enumerate}[(1)]
\item A subset $A$ of $X$ is said to be \emph{nowhere dense} if the closure $\overline{A}$ does not have an internal point.

\item A countable union of nowhere dense subsets is said to be \emph{meager}.

\item The topological space $X$ is called a \emph{Baire space} if $X$ is not meager in $X$.

\item A subset $A$ of $X$ is said to be \emph{almost open} or \emph{Baire measurable} if there is an open set $U$ in $X$ such that $A \triangle U = (U \setminus A) \cup (A \setminus U)$ is meager.
\end{enumerate}
\end{definition}

It follows from the Baire category theorem that a Polish space is a Baire space.

Clearly, every open subset of $X$ is almost open. The set of almost open subsets in $X$  forms  a $\sigma$-algebra. Therefore, every Borel subset of $X$ is almost open.

\begin{definition}
A topological group $G$ is a \emph{Baire group} if its underlying topological space is a Baire space.
\end{definition}

\begin{proposition} \label{proposition Bourbaki}
Let $G$ be a Baire group  and $A$ a subset of $G$. If $A$ is almost open and non-meager, then $A A^{-1}$ is a neighborhood of $e_G$.
\end{proposition}
\begin{proof}
This is \cite[Lemme 9 IX \S 6 $\mathrm{N}^\circ$  8]{Bourbaki} when $A$ is a Borel subset. To deduce the general case, it is sufficient to note that every almost open subset is a union of a $G_\delta$-set and a meager set.
\end{proof}

\begin{definition} \label{definition analytic subset}
Let $X$ be a Polish space. A subset $A$ of $X$ is called an \emph{analytic subset} if $A$ satisfies the following condition: There exist a Polish space $Y$, a Borel subset $B$ of $Y$ and a Borel measurable map $f \colon Y \to X$ such that $f(B) = A$.
\end{definition}

The following is a main tool in the proof of Theorem \ref{mtst_simple}.

\begin{theorem}[{\cite[Theorem 21.6]{Kechris}}]
Let $X$ be a Polish space. Then every analytic subset of $X$ is almost open.
\end{theorem}

The following proposition seems a folklore fact among experts, but we provide a proof of it for the convenience of the reader.

\begin{proposition} \label{proposition Polish fiber bundle}
Let $p \colon X \to Y$ be a fiber bundle over a Polish space $Y$ such that 
 $p^{-1}(y)$  is Polish for every 
 $y \in Y$.  Then $X$ is Polish.
\end{proposition}

To prove this, we use the following theorem:

\begin{theorem}[\cite{Michael}] \label{theorem Michael}
Let $f \colon X \to Y$ be an open continuous surjective map from a completely metrizable space $X$ to a paracompact Hausdorff space $Y$. Then $Y$ is completely metrizable.
\end{theorem}


\begin{proof}[Proof of Proposition \ref{proposition Polish fiber bundle}]
Since $Y$ is a Polish space, $Y$ is a paracompact Lindel\"of space. Hence $Y$ has a countable open covering $(U_n)_{n \in \NN}$ which satisfies the following:
\begin{enumerate}[(1)]
\item The open covering $(U_n)_{n \in \NN}$ is locally finite.

\item For every $n \in \NN$, there are $x_n \in U_n$ and a homeomorphism
\[ \varphi_n \colon U_n \times p^{-1}(x_n) \to p^{-1}(U).\]
\end{enumerate}

Since $U_n$ is an open subset of a Polish space $Y$, $U_n$ is a Polish space. Since a finite product of Polish spaces and a countable disjoint union of Polish spaces are Polish spaces, the space
\[ \widehat{X} = \coprod_n U_n \times p^{-1}(x_n)\]
is a Polish space. Define the map $\varphi \colon \widehat{X} \to X$ to be the sum of $\varphi_n$, i.e., the restriction of $\varphi$ to $U_n \times p^{-1}(x_n)$ coincides with $\varphi_n$. Then $\varphi$ is an open continuous map and 
 its
image is $X$. 
By Theorem \ref{theorem Michael} to complete the proof, it suffices to show that $X$ is paracompact and Hausdorff. We can easily show that $X$ is Hausdorff and omit the detail. Thus we only show the paracompactness of $X$.


Let $\mathcal{U}$ be an arbitrary open covering of $X$. We want to find a locally finite open covering $\mathcal{V}$ of $X$ which is a refinement of $\mathcal{U}$. For each $n \in \NN$, define the open covering $\mathcal{U}_n$ of $p^{-1}(U_n)$ by
\[ \mathcal{U}_n = \{ U \cap p^{-1}(U_n) \; | \; U \in \mathcal{U}\}.\]
Since $p^{-1}(U_n)$ is Polish, there is a locally finite open covering $\mathcal{V}_n$ of $p^{-1}(U_n)$ which is a refinement of $\mathcal{U}_n$. Then set
\[ \mathcal{V} = \bigcup_{n \in \NN} \mathcal{V}_n.\]
Since $(U_n)_{n \in \NN}$ is locally finite, we conclude that $\mathcal{V}$ is locally finite. This completes the proof.
\end{proof}

\subsection{Preliminaries of the separability of topological groups} \label{subsection separable groups}

Here we prove some facts of separability of topological groups we need in the proof of Theorem \ref{mtst_simple}. All of these are well-known or folklore facts
(for example, see \cite[Exercise 2.4.6]{Oh15a}).

\begin{lemma} \label{lemma separable}
Let $G$ be a connected topological group. Then the following are equivalent:
\begin{enumerate}[$(1)$]
\item $G$ is separable.

\item Every open neighborhood of $e_G$ is separable.

\item There is an open neighborhood $U$ of $e_G$ in $G$ such that $U$ is separable.
\end{enumerate}
\end{lemma}
\begin{proof}
The proof of $(1) \Rightarrow (2)$ is clear since every open subset of a separable space is separable. The implication $(2) \Rightarrow (3)$ is also clear. We show $(3) \Rightarrow (1)$. Let $U$ be a separable open neighborhood of  $e_G$ in  $G$. Let $K \subset U$ be a countable set such that $U \subset \overline{K}$. For a positive integer $n$, set
\[ U^n = \{ x_1 \cdots x_n \; | \; \textrm{$x_i \in U$ for every $i = 1, \cdots, n$}\}.\]
Note that $U^n \subset \overline{K^n}$ and hence that $U^n$ is also separable. Set
\[ V = \bigcup_{n \ge 1} U^n.\]
Then both $V$ and $G \setminus V$ are open. Since $G$ is connected, we have that $V = G$. Set $ K'  = \bigcup_{n \ge 1} K^n$.  Then we have that $U^n \subset \overline{K^n} \subset \overline{ K' }$ for every $n \ge 1$. This means that $G = \overline{ K' }$, and hence $G$ is separable. This completes the proof.
\end{proof}

\begin{corollary} \label{corollary separable 1}
Let $p \colon G \to H$ be a continuous homomorphism between connected topological groups such that $p$ is a covering map. Then $G$ is separable if and only if $H$ is separable.
\end{corollary}
\begin{proof}
In this case, there are an open neighborhood $U$ of $e_G$ in $G$ and an open neighborhood $V$ of $e_H$ in $H$ such that $p$ sends $U$ homeomorphically to $V$. If $G$ is separable, then Lemma \ref{lemma separable} implies that $U$ is separable. Hence $V$ is separable and therefore Lemma \ref{lemma separable} implies that $H$ is separable. The proof of the converse is similar, and the details are omitted.
\end{proof}

\begin{corollary} \label{corollary separable 2}
Let $p \colon G \to H$ be a continuous surjective homomorphism between connected topological groups. Assume that $H$ is separable and the kernel of $p$ is discrete. Then $\Ker(p)$ is countable.
\end{corollary}
\begin{proof}
In this case, $p$ is a covering map. Since $H$ is separable, Corollary \ref{corollary separable 1} implies that $G$ is separable. Hence there is an open neighborhood $U_x$ of $x \in \Ker(p)$ such that $x \ne y$ implies that $U_x \cap U_y = \emptyset$. Let $C$ be a countable dense set of $G$. Let $\varphi \colon \Ker(p) \to C$ be a map such that $\varphi(x) \in C \cap U_x$. Since $x \ne y$ implies $U_x \cap U_y =\emptyset$, $\varphi$ is injective. This shows that $\Ker(p)$ is countable.
\end{proof}

\begin{corollary} \label{corollary separable 3}
If $G$ is a separable connected topological group having the universal covering, then $\pi_1(G)$ is countable.
\end{corollary}
\begin{proof}
This immediately follows from Corollary \ref{corollary separable 2} that the kernel of the universal covering map $\tilde{G} \to G$ is naturally identified with $\pi_1(G)$.
\end{proof}

Many transformation groups are known to be locally contractible and hence have the universal coverings. Hence this corollary shows that fundamental groups of many transformation groups are countable.
%

\begin{remark}
The hypothesis in Corollary \ref{corollary separable 3} that $G$ has the universal covering cannot be eliminated. Indeed, consider the product group
\[ G = \prod_{n \in \NN} S^1.\]
Since the countable product of second countable spaces is second countable, $G$ is second countable and hence is separable. However, its fundamental group $\pi_1(G) = \ZZ^{\NN}$ is uncountable. Note that $G$ does not have the universal covering since it is not semilocally simply-connected (see Section 1.3 of \cite{Hatcher}).
\end{remark}

\subsection{Proof of Theorem \ref{mtst_simple}} \label{subsection proof of Polish groups}

The goal in this section is to prove Theorem \ref{mtst_simple}. We first prove the following:

\begin{lemma} \label{lemma open map}
Let $p \colon X \to Y$ be an open map and $B$ a subset of $Y$. Then the restriction $p_B \colon p^{-1}(B) \to B$  of $p$ to $p^{-1}(B)$ is an open map.
\end{lemma}
\begin{proof}
Let $U$ be an open subset of $p^{-1}(B)$. Then there is an open subset $V$ of $X$ such that $U = V \cap p^{-1}(B)$. Since $p$ is open, $p(V)$ is an open subset of $Y$. Therefore $p(U) = p(V \cap p^{-1}(B)) = p(V) \cap B$ is an open subset of $B$.
\end{proof}

In the proof below, we use the following notation.
Let $G$ be a group, $x$ an element of $G$, and $n$ a non-negative integer. Set
\[ U(G,x,n) = \{ x_1 \cdots x_n \mid \textrm{$x_i$ is conjugate to $x$ or $x^{-1}$}\}.\]

\begin{proof}[Proof of Theorem \ref{mtst_simple}]
We first show that $H$ is separable. Since every $H_n$ is separable, there is a countable dense subset $A_n$ of $H_n$. Then set $A = \bigcup_{n \ge 0} A_n$, and consider its closure $\overline{A}$ in $H$. Then $A$ is countable and $\overline{A}$ contain $H_n$ for every non-negative integer $n$. Since $H = \bigcup_{n \ge 0} H_n$, we have $\overline{A} = H$, and $H$ is separable.

Hence it follows from Corollaries \ref{corollary separable 1} and \ref{corollary separable 2} that $G$ is separable and $K$ is countable. Set $G_n = p^{-1}(H_n)$. Then $G_n$ is a covering space of $H_n$ with countable fiber $K$. Thus it follows from Proposition \ref{proposition Polish fiber bundle} that $G_n$ is a Polish group.

Let $N$ be a normal subgroup of $G$. Suppose that there is $x \in N$ but  $x \not\in K$.  It suffices to show $G =  N $. Let $k \in \NN$ so that $x \in G_k$. Since $H$ is simple and $p(x) \ne e_H$, we have
\[ H = \bigcup_{n=k}^\infty U(H, p(x), n).\]
Since $p(U(G,x,n)) = U(H, p(x), n)$, we have
\[ p\Big( \bigcup_{n=k}^\infty U(G,x,n) \Big) = H.\]
Hence we have
\[ G = \bigcup_{a \in K} a \cdot \Big( \bigcup_{n=k}^\infty U(G, x, n) \Big).\]
Since $K$ is countable, it follows from the Baire category theorem that for every $m \ge k$ there exist $l \ge m$ and $n \ge 0$ such that $U(G_l, x, n) \cap G_m$ is non-meager. Here we note that $U(G_l,x,n) \cap G_m$ is an analytic subset of $G_m$. Indeed, for $(\varepsilon_1, \cdots, \varepsilon_n) \in \{ -1, 0, 1\}^n$, define the map $\mu(\varepsilon_1, \cdots, \varepsilon_n)$ by
\[ \mu(\varepsilon_1, \cdots, \varepsilon_n) \colon G_l^n \to G, \quad (g_1, \cdots g_n) \mapsto g_1 x^{\varepsilon_1} g_1^{-1} \cdots g_n x^{\varepsilon_n} g_n^{-1}.\]
Then
\[ U(G_l, x, n) \cap G_m = \bigcup_{(\varepsilon_1, \cdots, \varepsilon_n) \in \{ 0, \pm1\}^n}\mu(\varepsilon_1, \cdots, \varepsilon_n) \big( \mu(\varepsilon_1, \cdots, \varepsilon_n)^{-1} (G_m) \big)\]
Since $G_m$ is a Polish subspace of $G_l$, Lemma \ref{lemma Polish facts} implies that $G_m$ is a $G_\delta$-set of $G_l$. Therefore $\mu(\varepsilon_1, \cdots, \varepsilon_n)^{-1}(G_m)$ is a $G_\delta$-set of $G_l^n$, and Lemma \ref{lemma Polish facts} again implies that $\mu(\varepsilon_1, \cdots, \varepsilon_n)^{-1}(G_m)$ is a Polish space. This means that $U(G_l, x,n) \cap G_m$ is an analytic set of $G_m$ and hence is almost open.


Hence we have showed that for every $m$ there exists $l \ge m$ and $n_0$ such that $U(G_l, x,  n_0 ) \cap G_m$ is non-meager in $G_m$ and analytic. Since
\[ \big( U(G_l, x, n_0) \cap G_m \big) \cdot \big( U(G_l, x, n_0) \cap G_m \big)^{-1} \subset U(G_l, x, 2n_0) \cap G_m,\]
Proposition \ref{proposition Bourbaki} shows that $U(G_l, x, 2n_0) \cap G_m$ is a neighborhood of $e_G$ in $G_m$. Hence
\[ \bigcup_{n\in \NN} U(G_l, x, n) \cap G_m \]
is a subgroup of $G_m$ containing a neighborhood of $e_G$ in $G_m$. Hence
\[ \bigcup_{n \in \NN} U(G, x, n) \cap G_m\]
is also a subgroup of $G_m$ containing a neighborhood of $e_G$ in $G_m$, and is an open subgroup of $G_m$.

Let $y \in G$. it follows from the condition (2) of Theorem \ref{mtst_simple} that there is $m$ such that $y$ is contained in the identity component $(G_m)_0$ of $G_m$. Since every open subgroup of a topological group contains the identity component of the whole group, we have that $y$ is contained in $U(G,x,n)$ for some $n$. Hence we have showed
\[ G = \bigcup_{n \in \NN} U(G,x,n).\]

Since $N$ is normal and $x \in N$, we have $U(G,x,n) \subset N$ for every $n$. Hence we have $G = N$, which completes the proof.
\end{proof}




\section{Relative simplicity of the universal covering of other transformation groups} \label{relative_simple_other}

In this section, we will use the Whitney $C^\infty$-topology on $\mathrm{Diff}^c(M)$ and its subgroups.

Similarly to Theorem \ref{closed_rel_simpl},
Theorem \ref{smooth_rel_simpl} can also be proved by  some modifications of a classical argument (for example, see \cite[Sections 2]{Ban97}). 
However, here we provide a proof using Theorem \ref{mtst_simple}.

\begin{proof}[Proof of Theorem \ref{smooth_rel_simpl}]
It is known that $\mathrm{Diff}_0^c(M)$ is simple (\cite{T74}, \cite{M74}, \cite{M75}, see also \cite{E84}, \cite[Theorem 2.1.1]{Ban97}). 

We also note that $H = \mathrm{Diff}_0^c(M)$ is known to be locally smoothly contractible (for example, see \cite[Prosition 1.2.1]{Ban97}) and so $\pi\colon \widetilde{\mathrm{Diff}_0^c}(M)\to \mathrm{Diff}_0^c(M)$ is genuinely a covering map and $G=\widetilde{\mathrm{Diff}_0^c}(M)$ is connected.
Hence, it is sufficient to prove the condition (2) of Theorem \ref{mtst_simple}.

For a compact subset $K$ of $M$,
let $C^{\infty}_K(M, M)$ be the space of $C^\infty$-maps from $M$ to $M$ whose supports are contained in $K$,
and let $\mathrm{Diff}_K(M)$ be the space of diffeomorphisms of $M$ whose supports are contained in $K$.
Recall that $C^\infty(M,M)$ is Polish (see next two paragraphs of \cite[Theorem 2.2.4]{Hirsch}). Since $C^\infty_K(M,M)$ is a closed subspace of $C^\infty(M,M)$, $C^\infty_K(M,M)$ is Polish.
Since $\mathrm{Diff}_K(M)$ is an open subset of $C^\infty_K(M,M)$ (see \cite[Theorem 2.1.7]{Hirsch}), by Lemma \ref{lemma Polish facts}, $\mathrm{Diff}_K(M)$ is a Polish space.

Let $(K_n)_n$ be an increasing sequence of compact subsets of $M$ such that
\[ M = \bigcup_n K_n.\]
and $K_{n-1}$ is contained in the interior of $K_n$.
Since $\mathrm{Diff}^c_0(M)$ is a closed subset of $\mathrm{Diff}^c(M)$, $H_n = \mathrm{Diff}_{K_n}(M) \cap \mathrm{Diff}^c_0(M)$ is a closed subgroup of the Polish group $\mathrm{Diff}_{K_n}(M)$. Let $p \colon \widetilde{\mathrm{Diff}^c_0}(M) \to \mathrm{Diff}^c_0(M)$ be the projection and set $G_n = p^{-1}(H_n)$. 
Let $\varphi \in G$.
Then  we see  that there is a path joining the identity to $\varphi$ which is contained in $G_m$ for some $m$. Thus Theorem~\ref{mtst_simple} completes the proof.
%
%
%
%
\end{proof}


We have a result similar to Theorems \ref{closed_rel_simpl} and \ref{smooth_rel_simpl} for the group for volume-preserving diffeomorphisms.
Before explaining that, we prepare some notations.

Let $M$ be an $m$-dimensional smooth manifold and $\Omega$ a volume form on $M$.
Let $\mathrm{Diff}^c(M,\Omega)$ denote the group of diffeomorphisms with compact support preserving the volume form $\Omega$.
Let $\mathrm{Diff}_0^c(M,\Omega)$ denote the identity component of $\mathrm{Diff}^c(M,\Omega)$ and
$\widetilde{\mathrm{Diff}}_0^c(M,\Omega)$ denote the universal cover of $\mathrm{Diff}_0^c(M,\Omega)$.
Then the {\it $($volume$)$ flux homomorphism} $\tFlux_\Omega \colon \mathrm{Diff}^c(M,\Omega) \to H^{m-1}(M;\RR)$ is defined by 
\[\tFlux_\Omega ([\{ \psi^t \}_{t \in [0,1]}]) = \int_0^1 [\iota_{X_t} \Omega] dt,\]
where $\iota$ is the interior product, $\{ \psi^t \}_{t \in [0,1]}$ is a path representing an element of $\mathrm{Diff}^c(M,\Omega)$ and $X_t$ is the time-dependent vector field generating the isotopy $\{ \psi^t \}_{t \in [0,1]}$.
It is known that the value $\tFlux_\Omega ([\{ \psi^t \}_{t \in [0,1]}])$ does not depend on the choice of the isotopy $\{ \psi^t \}_{t \in [0,1]}$ and thus the map $\tFlux_\Omega$ is a well-defined homomorphism (for example, see \cite[Theorem 3.1.1]{Ban97}).

Set $\Gamma_\Omega=\tFlux_\Omega(\pi_1(\Diff_0(M,\Omega)))$, which is called the \textit{flux group}.
Then, the flux homomorphism $\tFlux_\Omega \colon \mathrm{Diff}_0^c(M,\Omega) \to H^{m-1}(M;\RR)$ induces the homomorphisms $\Flux_\Omega \colon \mathrm{Diff}_0^c(M,\Omega) \to H^{m-1}(M;\RR)/\Gamma_\Omega$, which is also called the flux homomorphism.

\begin{thm}\label{vol_rel_simpl}
Let $M$ be a smooth manifold and $\Omega$ a volume form on $M$.
Let $G_\Omega = \Ker(\Flux_\Omega)$ and $N = \pi_1\left(G_\Omega\right)$.
Then, the universal covering $\tilde{G}_\Omega$ of $G_\Omega$ is simple relative to $N$.
\end{thm}

Similarly to Theorem \ref{closed_rel_simpl}, Theorem \ref{vol_rel_simpl} can also be proved by a minor modification of a classical argument (for example, see \cite[Sections 5]{Ban97}). 
However, here we provide a proof using Theorem \ref{mtst_simple}.

\begin{proof}[Proof of Theorem \ref{vol_rel_simpl}]

It is known that $G_\Omega$ is simple (\cite{T_un}, see also \cite[Theorem 5.1.3]{Ban97}). 

We also note that $G_\Omega$ is known to be locally smoothly contractible.
Indeed, by Moser's trick, it is known that
$\mathrm{Diff}_0^c(M,\Omega)$ is locally smoothly contractible
(for example, see \cite[Corollary 1.5.4]{Ban97}). 
Then, by an argument similar to the proof of  \cite[Corollary 6.3]{Fa}, we see that
$G_\Omega$ is also locally smoothly contractible and $\pi \colon \tilde{G}_\Omega\to G_\Omega$ is genuinely a covering map.
We also have $G_\Omega$ is connected.
Hence, it is sufficient to prove the condition (2) of Theorem \ref{mtst_simple}

For a compact subset $K$ of $M$, let $G_{\Omega,K} = G_\Omega \cap \mathrm{Diff}_K(M)$.
For a compact subset $K$ of $M$, recall that we have defined $\mathrm{Diff}_K(M)$ and proved that  $\mathrm{Diff}_K(M)$ is a Polish group in the proof of Theorem \ref{smooth_rel_simpl}. 

Here, we state that $G_{\Omega,K}$ is a Polish group.
Indeed, since $\mathrm{Diff}^c(M,\Omega)$ is a closed subset of $\mathrm{Diff}^c(M)$, $\mathrm{Diff}_K(M)\cap \mathrm{Diff}^c(M,\Omega)$ is a closed subset of $\mathrm{Diff}_K(M)$.
Since $\mathrm{Diff}_K(M)$ is a Polish space, $\mathrm{Diff}_K(M)\cap \mathrm{Diff}^c(M,\Omega)$ is a Polish space.
Since $\mathrm{Diff}^c_0(M,\Omega)$ is a closed subset of $\mathrm{Diff}^c(M,\Omega)$, $\mathrm{Diff}_K(M)\cap \mathrm{Diff}_0^c(M,\Omega)$ is a closed subset of $\mathrm{Diff}_K(M)\cap \mathrm{Diff}^c(M,\Omega)$.
Since $\mathrm{Diff}_K(M)\cap \mathrm{Diff}^c(M,\Omega)$ is a Polish space, $\mathrm{Diff}_K(M)\cap \mathrm{Diff}_0^c(M,\Omega)$ is a Polish space.
Since the flux homomorphism is known to be continuous, $G_{\Omega,K}$ is a closed subset of $\mathrm{Diff}_K(M)\cap \mathrm{Diff}_0^c(M,\Omega)$.
Since $\mathrm{Diff}_K(M)\cap \mathrm{Diff}_0^c(M,\Omega)$ is a Polish space, $G_{\Omega,K}$ is a Polish group.

Let $(K_n)_n$ be an increasing sequence of compact subsets of $M$ such that
\[ M = \bigcup_n K_n\]
and $K_{n-1}$ is contained in the interior of $K_n$. Set $H_n = \Diff_{K_n}(M, \Omega)$. Then by a similar consideration to the proof of Theorem~\ref{smooth_rel_simpl}, we can check that the condition (2) is satisfied and we conclude that Theorem~\ref{mtst_simple} completes the proof.
%
%
%
%
\end{proof}

\begin{remark}\label{rem_ono}
It is a natural question whether one can prove Theorem \ref{closed_rel_simpl} by using Theorem \ref{mtst_simple}.
Actually, it is possible.
However, we have to use the flux conjecture in order to prove that $\Ham(M,\omega)$ is locally path-connected.
The flux conjecture had been a long-standing open problem in symplectic geometry 
and finally proved by Ono \cite{On06} in a highly non-trivial way using the Floer theory.
\end{remark}

Here, we explain  what is $G/N$ in the setting of Theorem \ref{open_rel_simpl}.

Banyaga proved the following theorem.

\begin{thm}[{\cite[Th\'{e}or\`{e}me II.6.2]{Ban}}]\label{banyaga_open}
Let $(M, \omega)$ be a connected open symplectic manifold. 
Then the kernel $\mathrm{Ker}(\underline{\Cal}_M)$ of the homomorphism $\underline\Cal_M\colon \Ham(M,\omega)\to\RR/\Lambda_\omega$ is a simple group.
\end{thm}

\begin{prop}\label{open_rel}
Let $G$ and $N$ be groups in Theorem \ref{open_rel_simpl}.
Then, the group $G/N$ is isomorphic to the group $\mathrm{Ker}(\underline{\Cal}_M)$.
\end{prop}
\begin{proof}
The universal covering $\tHam(M,\omega) \to \Ham(M,\omega)$ induces the homomorphism $\mathrm{Ker}(\Cal_M) \to \mathrm{Ker}(\underline{\Cal}_M)$.
We can easily confirm that this map is surjective and its kernel is $N$.
\end{proof}

\begin{remark}\label{rem_kislev}
It is a natural question whether one prove Theorem \ref{open_rel_simpl} by using Theorems \ref{banyaga_open} and \ref{mtst_simple}.

Kislev \cite{Ki14} provided an example of an open symplectic manifold $(M,\omega)$ such that $\Lambda_\omega$ is a dense subset of $\RR$.
Then, $\mathrm{Ker}(\underline{\Cal}_M)$ is not locally path-connected and thus we cannot apply Theorem \ref{mtst_simple} in order to prove Theorem \ref{open_rel_simpl}.
\end{remark}
%
%

\section{Autonomous norm is well-defined} \label{well-def_aut}

In this section, we will prove Theorem \ref{auto_thm}.

Let $M$ be a smooth manifold and $G$ a subgroup of $\Diff(M)$.
An element $g$ of the universal covering $\tilde{G}$ of $G$ is called \textit{autonomous} if there exists a vector field $X$ with compact support such that
\begin{itemize}
    \item $\varphi_X^t \in G$ for every $t\in\RR$,
    \item $g=[\{\varphi_X^t\}_{t\in[0,1]}]$ in $\tilde{G}$.
\end{itemize}
Here, $\{\varphi_X^t\}_{t\in\RR}$ is a flow generated by $X$.

In this section, we often use the following lemma.
\begin{lem}\label{auton_lem}
    Let $g$ be an autonomous element of $\tilde{G}$ and let $h\in \tilde{G}$.
    Then, $g^{-1}$ and $hgh^{-1}$ are also autonomous elements of $\tilde{G}$.
\end{lem}

\begin{proof}
    Let $X$ be a vector field generating $g$.
    Then, the vector fields $-X$, $h_*X$ generate $g^{-1}$, $hgh^{-1}$, respectively.
    Thus, $g^{-1}$ and $hgh^{-1}$ are also autonomous elements of $\tilde{G}$.
\end{proof}

\begin{proof}[Proof of Theorem \ref{auto_thm} when $M$ is closed]
We can easily take an autonomous element $g$ of $\tHam(M,\omega)$ such that $g\notin \pi_1(\Ham(M,\omega))$.
Then, by Theorem \ref{closed_rel_simpl}, every $f \in \tHam(M,\Omega)$ can be represented as a product of conjugates of $g$ or $g^{-1}$.
By Lemma \ref{auton_lem}, conjugates of $g$ or $g^{-1}$ are autonomous.
\end{proof}

\begin{proof}[Proof of Theorem \ref{auto_thm} when $M$ is open]
Take $f \in \tHam(M,\omega)$ and set $a=\Cal_M(f)$.
Then, one can easily take a Hamiltonian $H\colon M\to \RR$ such that $\Cal_M(\tilde\varphi_H)=a$ where $\tilde\varphi_H = [\{\varphi_H^t\}_{t\in[0,1]}] \in \tHam(M,\omega)$.
Then, $f\circ(\tilde\varphi_H)^{-1} \in \Ker(\Cal_M)$.

We can easily take an autonomous element $g$ of $\Ker(\Cal_M)$ such that $g\notin \pi_1(\Ham(M,\omega))$.
Then, by Theorem \ref{open_rel_simpl}, $f\circ(\tilde\varphi_H)^{-1} \in \Ker(\Cal_M)$ can be represented as a product of conjugates of $g$ or $g^{-1}$.
By Lemma \ref{auton_lem}, conjugates of $g$ or $g^{-1}$ are autonomous.

Since $\tilde\varphi_H$ is autonomous, $f$ can be represented as a product of autonomous elements of $\widetilde{\mathrm{Diff}_0^c}(M,\Omega)$.
\end{proof}

Moreover, we can prove the following theorems.

\begin{thm}\label{smooth_auto}
    Let $M$ be a smooth manifold.
Every $f \in \widetilde{\mathrm{Diff}_0^c}(M)$ can be represented as a product of autonomous elements of $\widetilde{\mathrm{Diff}_0^c}(M)$.
\end{thm}

\begin{thm}\label{vol_auto}
Let $M$ be an $n$-dimensional smooth manifold and $\Omega$ a volume form on $M$.
Every $f \in \widetilde{\mathrm{Diff}_0^c}(M,\Omega)$ can be represented as a product of autonomous elements of 
$\widetilde{\mathrm{Diff}_0^c}(M,\Omega)$.
\end{thm}

If we use Theorem \ref{smooth_rel_simpl} instead of Theorem \ref{closed_rel_simpl} in the proof of Theorem \ref{auto_thm}, we can prove Theorem \ref{smooth_auto} similarly.
Hence we omit the proof of Theorem \ref{smooth_auto}.

In order to prove Theorem \ref{vol_auto}, we use the following lemma.

\begin{lem}[{for example, see the proof of \cite[Proposition 3.1.6]{Ban97}}]\label{vol_flux_surj}
Let $M$ be an $n$-dimensional smooth manifold and $\Omega$ a volume form on $M$.
For every $a\in H^{n-1}(M;\RR)$, there exists a vector field $X$ such that
\begin{itemize}
    \item $\varphi_X^t \in \Diff(M,\Omega)$ for every $t\in\RR$,
    \item $\Flux_\Omega([\{\varphi_X^t\}_{t\in[0,1]}])=a$.
\end{itemize}
\end{lem}

\begin{proof}[Proof of Theorem \ref{vol_auto}]
Take $f \in \widetilde{\mathrm{Diff}_0^c}(M)$ and set $a=\Flux(f)$.
By Lemma \ref{vol_flux_surj}, there exists a vector field $X$ such that
\begin{itemize}
    \item $\varphi_X^t \in \Diff(M,\Omega)$ for every $t\in\RR$,
    \item $\Flux_\Omega(\tilde\varphi_X)=a$, where $\tilde\varphi_X=[\{\varphi_X^t\}_{t\in[0,1]}] \in \widetilde{\mathrm{Diff}_0^c}(M,\Omega)$.
\end{itemize}
Then, $f\circ(\tilde\varphi_X)^{-1} \in \Ker(\Flux_\Omega)$.

We can easily take an autonomous element $g$ of $\Ker(\Flux_\Omega)$ such that $g\notin \pi_1(\Ker(\Flux_\Omega))$.
Then, by Theorem \ref{vol_rel_simpl}, $f\circ(\tilde\varphi_X)^{-1} \in \Ker(\Flux_\Omega)$ can be represented as a product of conjugates of $g$ or $g^{-1}$.
By Lemma \ref{auton_lem}, conjugates of $g$ or $g^{-1}$ are autonomous.

Since $\tilde\varphi_X$ is autonomous, $f$ can be represented as a product of autonomous elements of 
$\widetilde{\mathrm{Diff}_0^c}(M,\Omega)$.
\end{proof}



\section{The Tsuboi metric of $\tHam(M,\omega)$} \label{tsuboi_metric}
In this section, we prove that
the Tsuboi metric space $\TK(\tHam(M,\omega))$ of $\tHam(M,\omega)$ is not quasi-isometric to $\RR_{\ge 0}$ 
(Theorem \ref{general ishida}). 
The strategy of the proof is the same as that of Ishida \cite{I18}.

Let $(M,\omega)$ be a  $2N$-dimensional  symplectic manifold.
Take a symplectic embedding $\iota \colon (\mathbb{B}^{2N}(R),\omega_0) \to (M,\omega)$ 
so that $U:=\iota(\mathbb{B}^{2N}(R))$ is displaceable.
Here, $\mathbb{B}^{2N}(R):=\{ x \in \RR^{2N} \mid \| x \| < R \} $ denotes the symplectic open ball with radius $R$  and $\omega_0$ is the standard symplectic form on $\mathbb{B}^{2N}(R)$.
Let $\iota^{-1} \colon U \to \mathbb{B}^{2N}(R) $ denote the inverse map of $\iota \colon \mathbb{B}^{2N}(R) \to U$.
Recall that $p_{U,M} \colon \tHam(U,\omega|_U) \to \tHam(M,\omega)$ denotes the homomorphism induced by the inclusion $U \hookrightarrow M$ and $\tHam_U(M,\omega)$ denotes the image of $p_{U,M}$. 
 The map $p_{U,M}$ induces the map $\underline{p}_{U,M} \colon \Ham(U,\omega|_U) \to \Ham(M,\omega)$ and let $\Ham_U(M,\omega)$ denote the image of $\underline{p}_{U,M}$. 
For $r>1$, let $\psi_r \colon \mathbb{B}^{2N}(R) \to \mathbb{B}^{2N}(R/r)$ be the map defined by $\psi_r(x)=x/r$.
We define the homomorphism $s_r \colon \Ham_U(M,\omega) \to \Ham_U(M,\omega)$ by 
\[
s_r(f)(x)=
\begin{cases}
 (\iota \circ \psi_r \circ \iota^{-1})  \circ g \circ (\iota \circ \psi_r^{-1} \circ \iota^{-1})  (x) & \text{if $ \iota^{-1}(x)  \in \mathbb{B}^{2N}(R/r)$}, \\
x & \text{if $ \iota^{-1}(x)  \not\in \mathbb{B}^{2N}(R/r)$},
\end{cases}
\]
 for $f \in \Ham_U(M,\omega)$, where $g \in  \Ham(U,\omega|_U) $ is the map with $ \underline{p}_{U,M} (g)=f$. 
We define $\tilde{s}_r \colon \tHam_U(M,\omega) \to \tHam_U(M,\omega)$ 
by $\tilde{s}_r([\{f_t\}])=[s_r(f_t)]$.

\begin{lemma} \label{shrinking lem}
For every partial Calabi quasimorphism $\mu \colon \tHam(M,\omega) \to \RR$, $f \in \tHam_U(M,\omega)$ and $r >1$, we have 
\[ \mu(\ts_r(f)) = r^{-(\dim M+  2 )} \mu(f). \]
\end{lemma}
\begin{proof}
Take $g \in \tHam(U,\omega|_U)$ such that $p_{U,M}(g)=f$.
Since $U$ is displaceable and $\mu$ has the Calabi property, we have 
$\mu(f)=\mu(p_{U,M}(g))=\Cal_U(g)$. 
Let $H \colon [0,1] \times  U  \to \RR$ be the Hamiltonian and $\{\varphi_H^t\}$ the Hamiltonian isotopy associated to $H$ such that $[\{\varphi_H^t\}]=g$.
Take $g' \in \tHam(U,\omega|_U)$ such that $p_{U,M}(g')=\ts_r(f)$. Then, $g'$ is generated by the Hamiltonian $ H^r  \colon [0,1] 
\times   U   \to \RR$ defined by 
 \[H^r(t,x)=
\begin{cases}
\frac{1}{r^2}H(t, (\iota \circ \psi_r^{-1} \circ \iota^{-1})(x) ) & \text{if $ \iota^{-1}(x)  \in \mathbb{B}^{2N}(R/r)$}, \\
0 & \text{if $ \iota^{-1}(x)  \not\in \mathbb{B}^{2N}(R/r)$}.
\end{cases}
\] 
Hence, we have
\begin{align*}
\mu( \ts_r(f) ) = \Cal_U(g')  ={} &   \int_0^1\int_{x\in\mathbb{B}^{2N}(R/r)} 
 H_t^ r (\iota(x)) \omega_{\rm std}^N\,dt \\
 ={} & r^{- 2 }\int_0^1\int_{y\in\mathbb{B}^{2N}(R)} H_t(\iota(y)) |D\psi_r| \omega_{\rm std}^N\,dt \\
 ={} &  r^{-( 2N+ 2 )}\Cal_U(g)=r^{-( 2N+ 2 )}\mu(f). \qedhere
\end{align*}
\end{proof}

Let $\dg$ (resp. $\du$) denote the Tsuboi metric on $\TK(\tHam(M,\omega))$ (resp. $\TK(p_{U,M}(\Ker(\Cal_U)))$).

\begin{lemma} \label{ishida lem}
 Let $r > 1$ be a real number greater than 1.
\begin{enumerate}
    \item For every $f \in\tHam_U(M,\omega) \setminus \pi_1(\Ham(M,\omega))$ and $m,n \in \NN$, 
    \[\dg([\ts^m_r (f)], [\ts^n_r (f)]) \geq (\dim M \log r)|m - n|.\] 
    \item For every $f \in p_{U,M}(\Ker(\Cal_U)) \setminus \pi_1(\Ham(M,\omega))$ and $n \in \NN$, 
    \[\du([\ts^n_r (f)], [\ts^{n+1}_r (f)]) \leq \du([f], [\ts_r(f)]).\]
\end{enumerate}
\end{lemma}

\begin{proof}
Assume that $m < n$.
Since the volume of the support of $\pi(\ts^m_r (f)) = s^m_r (\pi(f))$ is just  $r^{(n-m) \dim M }$ 
times of that of $s^n_r (\pi(f))$, we have  $q_{\ts^n_r(f)}(\ts^m_r (f)) \geq r^{(n-m)\dim M}$. 
This implies (1).

By Theorem \ref{open_rel_simpl}, $\ts_r(f)$ is written as a product
\[\ts_r(f) = (h_1 f^{\ee_1} h_1^{-1}) \cdots (h_k f^{\ee_k} h_k^{-1}), \]
where each $\ee_i$ is $1$ or $-1$ and $h_i \in p_{U,M}(\Ker(\Cal_U))$. 
Then we have 
\[\ts_r^{n+1}(f) = (h_1 \ts_r^n(f)^{\ee_1} h_1^{-1}) \cdots (h_k \ts_r^n(f)^{\ee_k} h_k^{-1}), \]
and thus $q_{s_r^n (f)}(s_r^{n+1} (f)) \leq  q_{f}(s_r(f))$. 
The inequality $q_{s_r^{n+1}(f)}(s_r^n (f)) \leq q_{s_r(f)}(f)$ similarly follows. Hence we have (2).
\end{proof}

\begin{proposition} \label{ishida ineq}
For every $f \in p_{U,M}(\Ker(\Cal_U))\setminus \pi_1(\Ham(M,\omega))$, there exist a sequence $\{f_n\}$ in $\tHam(M,\omega)$ with $f_0=f$  and $g \in \tHam(M,\omega)$, positive constants $C_1,C_2$, and a constant $C_3$ satisfying the following.
\begin{enumerate}
    \item $d([f_n], [f_m]) \geq C_1|n - m|$,
    \item $d([f_n], [f_{n+1}]) \leq  C_2$,
    \item $d([f_n], [g^m]) \geq \log m + C_3$.
\end{enumerate}
\end{proposition}

\begin{proof}
Take $r >1$ and  set $f_n= \ts_r^n(f)$. Then (1) and (2) follow from Lemma \ref{ishida lem} 
(note that $\dg([h],[h']) \leq \du([h],[h'])$ for every $h,h' \in  p_{U,M}(\Ker(\Cal_U))\setminus \pi_1(\Ham(M,\omega))$).
Take a partial Calabi quasimorphism $\mu \colon \tHam(M,\omega) \to \RR$ (it exists; see Remark \ref{calabi qm remark}) and $g \in \tHam(M,\omega) \setminus \pi_1(\Ham(M,\omega))$ such that $\mu(g)\neq 0$. 
Set $k=q_{f_n}(g^m)$. 
Then,  by Theorem \ref{closed_rel_simpl}, we can write $g^m$ as
\[ g^m= h_1 \cdots h_k, \]
where $h_i$ is a conjugate of an element of $f_n$ or $f_n^{-1}$ for $i=1,\dots,k$.
Let $C$ be the constant with respect to the partial quasimorphism $\mu$ (see Definition \ref{def_of_PCQ} (2)).
Since $\| h_i \|_U=1$, we have 
\[\left| \mu(g^m)-\sum_{i=1}^k\mu(h_i) \right| \leq C k.\]
Since $\mu(\ts_r(h)) = r^{-(\dim M+ 2 )} \mu(h)$ for $h \in \tHam_U(M,\omega)$ by Lemma \ref{shrinking lem}, we observe that $|\mu(h_i)| 
=|\mu(f_n)|  \leq |\mu(f)| $. Thus, we obtain 
\[ |\mu(g^m)| \leq (C+|\mu(f)|)k.\]
This implies 
\[d([f_n],[g^m]) \geq \log\frac{|\mu(g^m)|}{C+|\mu(f)|}
=  \log m + \log\frac{|\mu(g)|}{C+|\mu(f)|} \]
and hence (3) holds.
\end{proof}

\begin{proof}[Proof of Theorem \ref{general ishida}]
Assume that $\TK(\tHam(M,\omega))$ is quasi-isometric to the half line (with the standard metric). Then there exists a quasi-isometric embedding $\Phi \colon \TK(\tHam(M,\omega)) \to \RR_{>0}$.
By Proposition \ref{ishida ineq} (3), we have $\Phi([f])< \Phi([g^m])$ for sufficiently large $m \in \NN$. By Proposition \ref{ishida ineq} (1), there exists $n \in \NN$ such that $\Phi([g^m])<\Phi([f_n])$; take such a smallest $n=n_m$. 
By Proposition \ref{ishida ineq} (2) and the definition of $n_m$,  the function $m\mapsto \Phi([f_{n_m}]) - \Phi([g^m])$ is bounded, but this contradicts Proposition \ref{ishida ineq} (3).
\end{proof}

\section{The Tsuboi metric of $S_\infty$} \label{symmetric_group}
In this section, we consider the symmetric groups and the alternating group.
For a positive integer $n$, let $S_n$ denote the symmetric group of degree $n$.
We define the symmetric group $S_\infty$ to be $\bigcup_{n=1}^\infty S_n$.
More precisely, we define $S_\infty$ as the direct limit of the direct system induced from the standard injections $S_n\to S_{n+1}$.
We define the infinite alternating group $A_n$, $A_\infty$ similarly.

The goal of this section is to show the 
 Tsuboi metric space $\TK(S_\infty)$ of $S_\infty$ 
is quasi-isometric to $\RR_{\ge 0}$ (Theorem \ref{theorem S_infinity_intro}).

We introduce the following notation. 
For a positive integer $k$, set
\[ \iota_k = (1,2) (3,4)\cdots (2k-1,2k). \]
We start with the following criterion, which will frequently be used in this section.

\begin{lemma} \label{lemma three}
Let $l$ and $k$ be positive integers with $l \le k$. Assume that $n \ge 2k$ or $n = \infty$. If $l$ and $k$ have the same parity, then $\iota_l$ is a product of three elements of $S_n$ which are conjugate to $\iota_k$.
\end{lemma}
\begin{proof}
Let $s$ be an integer such that $k - l = 2s$. Let $x_1, y_1, \cdots, x_{2k}, y_{2k}$ be distinct elements in $\{ 1, \cdots, n\}$. Define $\sigma, \sigma', \sigma'' \in S_n$ by
\begin{align*}
\sigma & = (x_1, y_1) \cdots (x_{2k}, y_{2k}) \\
& = \big( (x_1, y_1) \cdots (x_l, y_l) \big) \cdot \big( (x_{l+1}, y_{l + 1})(x_{l + 2}, y_{l + 2}) \big) \cdots \big( (x_{l + 2s - 1}, y_{l + 2s - 1})(x_{l + 2s}, y_{l + 2s}) \big),
\end{align*}
\begin{align*}
\sigma' = \big( (x_1, y_1) \cdots (x_l, y_l) \big) \cdot \big( (x_{l+1}, x_{l + 2})(y_{l + 1}, y_{l + 2}) \big) \cdots \big( (x_{l + 2s - 1}, x_{l + 2s})(y_{l + 2s-1}, y_{l + 2s}) \big),
\end{align*}
\begin{align*}
\sigma'' = \big( (x_1, y_1) \cdots (x_l, y_l) \big) \cdot \big( (x_{l+1}, y_{l+2}) (x_{l+2}, y_{l+1}) \big) \cdots \big( (x_{l + 2s - 1}, y_{l + 2s}) (x_{l+ 2s}, y_{l + 2s - 1}) \big).
\end{align*}
Then we have
\[ \sigma \sigma' \sigma'' = (x_1, y_1) \cdots (x_l, y_l),  
\]
which is conjugate to $\iota_{l}$. 
 Since $\sigma$, $\sigma'$, and $\sigma''$ are conjugate to $\iota_k$, this completes the proof.
\end{proof}

\begin{lemma} \label{lemma S relatively simple}
Let $n$ be an integer greater than $1$  or $\infty$. Then $A_n$ is the maximum normal subgroup of $S_n$. In particular, $S_n$ is relatively simple.
\end{lemma}
\begin{proof}
The case that $n = 2$ is obvious. Suppose that $n \ge 3$ and $n \ne 4$. Let $\sigma \in S_n \setminus A_n$. We want to show that $\langle\langle\sigma \rangle\rangle = S_n$, where $\langle\langle\sigma \rangle\rangle$ is the normal subgroup generated by $\{\sigma\}$.

 Note that  $\sigma^2 \in \langle\langle \sigma \rangle\rangle \cap A_n$. If $\sigma^2 \ne 1$, then $A_n \subsetneq \langle\langle \sigma \rangle\rangle$ since $A_n$ is simple, and hence $\langle\langle \sigma \rangle\rangle = S_n$. If $\sigma^2 = 1$, then $\sigma$ is a product of transpositions whose supports are disjoint. Since $\sigma$ is odd, Lemma \ref{lemma three} implies that a transposition belongs to $\langle\langle \sigma \rangle\rangle$. It is well known that a transposition normally generates $S_n$, and hence $\langle\langle \sigma \rangle\rangle = S_n$.

Finally, we show that $A_4$ is the maximum normal subgroup of $S_4$. Let $\sigma \in S_4 \setminus A_4$. Then $\sigma$ is a transposition or a cyclic permutation with length $4$. If $\sigma$ is a transposition then $\langle\langle \sigma \rangle\rangle = S_4$. If $\sigma$ is a cyclic permutation with length $4$,
\[ (1,2,3,4)^2 (1,3,2,4) = (3,4)\]
implies that $\langle\langle \sigma \rangle\rangle$ contains a transposition, and hence we have $S_4$.
\end{proof}

In the rest of this section we show Theorem \ref{theorem S_infinity_intro}, which is restated below;

\begin{theorem} \label{theorem S_infinity}
The metric space $(\TK(S_\infty), d)$ is quasi-isometric to the half line $\RR_{\ge 0}$.
\end{theorem}

Define the subset $X$ of $\TK(S_\infty)$ by
\[ X = \{ [\iota_{3^k}] \; | \; \textrm{$k$ is a non-negative integer}\}.\]
Then Theorem \ref{theorem S_infinity} is deduced from the following two propositions:

\begin{proposition} \label{proposition coarsely dense}
$X$ is coarsely dense in $\TK(S_\infty)$. Namely, there exists $R \ge 0$ such that for every $\sigma \in S_\infty$, there is $x \in X$ such that $d([\sigma], x) \le R$.
\end{proposition}

\begin{proposition} \label{proposition X}
The metric space $X$ is quasi-isometric to the half line $\RR_{\ge 0}$.
\end{proposition}

\subsection{Proof of Proposition \ref{proposition X}}

We first introduce the following function $\nu \colon S_\infty \to \RR_{\ge 0}$, which is useful in the proof of Theorem \ref{theorem S_infinity}. Recall that the support of $\sigma \in S_\infty$ is the set
\[ \supp(\sigma) = \{ x \in \{ 1,2,3, \cdots \} \; | \; \sigma(i) \ne i\}.\]
Define $\nu \colon S_\infty \to \RR_{\ge 0}$ by
\begin{equation}\label{symmetric support}
\nu(\sigma) = \# \supp(\sigma).
\end{equation}
Then $\nu$ is a conjugation-invariant norm on $S_\infty$.

We use the following notation: Let $G$ be a relatively simple group and let $[x], [y] \in \TK(G)$. Set $d_{as}([x], [y]) = \log q_x(y)$. Then $d_{as}$ satisfies the triangle inequality.






\begin{lemma} \label{lemma asymmetry}
Let $m$ and $n$ be odd integers greater than 1. Then the following hold  for every $m,n$:
\begin{enumerate}[(1)]
\item $d_{as}([\iota_{mn}], [\iota_n]) = \log 3$.

\item $d_{as}([\iota_n], [\iota_{mn}]) =  \log m$.
\end{enumerate}
\end{lemma}
\begin{proof}
We first show (1). By Lemma \ref{lemma three}, we have that $d_{as}([\iota_{mn}], [\iota_{n}]) \le \log 3$. Since $\iota_{mn}$ and $\iota_{n}$ are odd permutations and are not conjugate, we need at least three elements to express $\iota_n$ as a product of elements which are conjugate to $\iota_{mn}$.

Next we show (2). It is clear that $\iota_{mn}$ is a product of $m$ elements which are conjugate to $\iota_{n}$. Since $\nu(\iota_{mn}) = 2 mn$, $\nu(\iota_{n}) = 2 n$ and $\nu$ is a conjugation invariant norm, we have that we need $m$ elements to express $\iota_{mn}$ as a product of conjugates of $\iota_{n}$.
\end{proof}

Lemma \ref{lemma asymmetry} immediately implies the following:

\begin{corollary} \label{corollary skeleton}
Let $l$ and $k$ be positive integers. Then we have $d([\iota_{3^k}], [\iota_{3^k}]) = |k - l| \cdot \log 3$.
\end{corollary}

\begin{proof}[Proof of Proposition \ref{proposition X}]
By Corollary \ref{corollary skeleton}, $X$ is isometric to $(\log 3) \cdot \ZZ_{\ge 0}$. Since $(\log 3) \cdot \ZZ_{\ge 0}$ and $\RR_{\ge 0}$ are quasi-isometric, this completes the proof.
\end{proof}

\subsection{Proof of Proposition \ref{proposition coarsely dense}}

Let $\gamma_n$ denote the cyclic permutation: 
\[\gamma_n = (1,2, \cdots, n).\]

\begin{lemma} \label{lemma key 1S}
Let $n$ be an integer greater than $1$. Then there exist $k \ge  \Big\lceil \dfrac{n}{6} \Big\rceil$ and $\sigma, \sigma', \sigma'' \in S_\infty$ which are conjugate to $\gamma_n$ such that
\[ \iota_k = \sigma \sigma' \sigma''.\]
\end{lemma}

To see this, we use the following lemma by 
the third author.

\begin{lemma}[{see \cite[Lemma 5]{K11}}]\label{lemma kodama 1}
Let $n$ be an integer at least $3$. Let $k = \Big\lfloor \dfrac{n}{3} \Big\rfloor$. Then there exist elements $\sigma$ and $\sigma'$ which are conjugate to $\gamma_n$ such that
\[ \iota_k = \sigma \sigma'.\]
\end{lemma}

\begin{proof}[Proof of Lemma \ref{lemma key 1S}]
If $n$ is odd, then $\gamma_n^2$ and $\gamma_n$ are conjugate and $\lceil n/6 \rceil \le \lfloor n/3 \rfloor$ for $n \ge 3$, the proof is completed by Lemma \ref{lemma kodama 1}.  
If $n=2$, we can take $k=3$ since $\iota_3 = (1,2)(3,4)(5,6)$ and each transposition is conjugate to $\gamma_2$. 
Suppose that $n$ is even and 
 $n = 2m\geq 4$.
Since
\[ (3,1, 5,6,7 \cdots, 2m+2) (1,2,3, \cdots, 2m) = (1,2)(3,4,6,8, \cdots ,2m, 5, 7, \cdots, 2m+1, 2m + 2),\]
we have that the disjoint union of $\gamma_n$ and $\iota_1$ is a product of two elements conjugate to $\gamma_n$. Hence Lemma \ref{lemma kodama 1} completes the proof.
\end{proof}

\begin{lemma} \label{lemma 11.10}
Let $\sigma \in S_\infty$. Then there exist $\sigma^{(1)}$, $\sigma^{(2)}$ and $\sigma^{(3)}$ which are conjugate to $\sigma$ and an integer $k \ge \lceil \nu(\sigma) / 6 \rceil$ such that
\[ \iota_k = \sigma^{(1)} \sigma^{(2)} \sigma^{(3)}.\]
\end{lemma}
\begin{proof}
Let $\sigma_1, \cdots, \sigma_s$ be cyclic permutations with disjoint supports such that $\sigma = \sigma_1 \cdots \sigma_s$.  By Lemma \ref{lemma key 1S},  for each $i$, there exists an integer $k_i \ge \lceil \nu(\sigma_i) / 6 \rceil$ such that $\iota_{k_i}$ is a product of three elements which are conjugate to $\sigma_i$. Then we have
\[ \Big\lceil \frac{\nu(\sigma)}{6} \Big\rceil \le \Big\lceil \frac{\nu(\sigma_1)}{6} \Big\rceil + \cdots + \Big\lceil \frac{\nu(\sigma_s)}{6} \Big\rceil \le k_1 + \cdots + k_s = k.\]
This completes the proof.
\end{proof}

Next we consider how we express a cyclic permutation as a product of conjugates of $\iota_{k}$.

\begin{lemma}[{see \cite[Sublemma 8]{K11}}] \label{lemma kodama 2}
The following hold:
\begin{enumerate}[(1)]
\item $\gamma_{2n+2}$ is a product of a conjugate of $\iota_n$ and a conjugate of $\iota_{n+1}$.

\item $\gamma_{2n+1}$ is a product of two elements which are conjugate to $\iota_n$
\end{enumerate}
\end{lemma}

\begin{lemma} \label{lemma key 2S}
Let $n$ be an integer greater than $1$. Then there exists a positive integer $k$ such that $k \le n$ and $\gamma_n$ is a product of three elements which are conjugate to $\iota_k$.
\end{lemma}
\begin{proof}
Suppose that $n$ is odd and set $n = 2m + 1$. Then $\gamma_{2m+1}$ is a product of two elements $\iota^{(1)}$ and $\iota^{(2)}$ which are conjugate to $\iota_m$ ((2) of Lemma \ref{lemma kodama 2}). Let $x_1, \cdots, x_m, y_1, \cdots, y_m, z_1, \cdots, z_m, w_1, \cdots, w_m$ be distinct positive integers which are not contained in $\supp(\iota^{(1)}) \cup \supp(\iota^{(2)})$. 
 Set 
\[ \iota = \iota^{(1)}(x_1,y_1) (x_2, y_2) \cdots (x_m, y_m), \quad \iota' = \iota^{(2)}(z_1, w_1) (z_2, w_2) \cdots (z_m, w_m),\]
\[ \iota'' = (x_1,y_1) \cdots (x_m, y_m) (z_1, w_1) \cdots (z_m, w_m).\]
Then these are conjugate to $\iota_{2m}$ and we have
\[ \iota \iota' \iota'' = \iota^{(1)} \iota^{(2)} = {\gamma}_{2m+1}.\]

Next suppose that $n$ is even and set $n = 2m+2$. 
 The case $m=0$ is obvious: $\gamma_2=\iota_1=\iota_1^3$. Let $m \geq 1$. 
Then $\gamma_{2m+2}$ is a product $\iota^{(1)} \iota^{(2)}$ of $\iota^{(1)}$ and $\iota^{(2)}$, where $\iota^{(1)}$ is conjugate to $\iota_m$ and $\iota^{(2)}$ is a conjugate of $\iota_{m+1}$ ((1) of Lemma \ref{lemma kodama 2}). 
Let $x_1, \cdots, x_{m+1}$, $y_1, \cdots, y_{m+1}$, $z_1, \cdots, z_m$, $w_1, \cdots, w_m$
be distinct positive integers which are not contained in $\supp(\iota^{(1)}) \cup \supp(\iota^{(2)})$. 
Then set
\[ \iota = \iota^{(1)} (x_1, y_1) \cdots (x_{m+1}, y_{m+1}),\quad \iota' = \iota^{(2)} (z_1, w_1) \cdots (z_m, w_m), \]
\[ \iota'' = (x_1,y_1) \cdots (x_{m+1}, y_{m+1}) (z_1, w_1) \cdots (z_m, w_m).\]
These are conjugate to  $\iota_{2m+1}$  and we have
\[ \iota \iota' \iota'' = \iota^{(1)} \iota^{(2)} = {\gamma}_{2m+2}.\]
This completes the proof.
\end{proof}

By Lemma \ref{lemma key 2S}, we have the following:

\begin{lemma} \label{lemma 11.13}
For every $\sigma \in S_\infty$, there exists a positive integer $k$ such that $k \le \nu(\sigma)$ and $\sigma$ is a product of three elements which are conjugate to $\iota_k$.
\end{lemma}

Define a subset $X'$ of $\TK(S_\infty)$ by
\[ X' = \{ [\iota_k] \; | \;\textrm{$k$ is a positive odd number}\}.\]

\begin{lemma} \label{lemma coarsely dense}
$X$ is coarsely dense in $X'$.
\end{lemma}
\begin{proof}
Let $k$ be a positive odd integer, and let $l$ be a non-negative integer such that $3^l \le k < 3^{l+1}$. Then Lemma \ref{lemma three} implies that $d_{as}([\iota_k], [\iota_{3^l}]) \le \log 3$. 
Lemmas \ref{lemma three} and \ref{lemma asymmetry} imply that
\[ d_{as}([\iota_{3^l}], [\iota_k]) \le d_{as}([\iota_{3^l}], [\iota_{3^{l+1}}]) + d_{as}([\iota_{3^{l+1}}], [\iota_k]) \le 2 \log 3.\]
This completes the proof.
\end{proof}

Now we are ready to prove Proposition \ref{proposition coarsely dense}.

\begin{proof}[Proof of Proposition \ref{proposition coarsely dense}]
By Lemma \ref{lemma coarsely dense}, it suffices to show that $X'$ is coarsely dense in $S_\infty \setminus A_\infty$. Let $\sigma \in S_\infty \setminus A_\infty$. 
By Lemma \ref{lemma 11.10}, there exists a positive integer $k_0 \ge \lceil \nu(\sigma) / 6 \rceil$ such that $\iota_{k_0}$ is a product of three elements which are conjugate to $\sigma$. Since $\sigma$ is an odd permutation, $k_0$ is odd and we have $\iota_{k_0} \in X$. Hence we have
\[ d_{as}([\sigma], [\iota_{k_0}]) \le \log 3.\]

On the other hand, by Lemma \ref{lemma 11.13} there exists a positive integer $k_1 \le \nu(\sigma)$ such that $\sigma$ is a product of three elements which are conjugate to $\iota_{k_1}$. Since $\sigma$ is an odd permutation, $k_1$ is odd and we have $\iota_{k_1} \in X'$. Note $k_1 \le 9 k_0$. By the triangle inequality of $d_{as}$, we have
\begin{align*}
d_{as}([\iota_{k_0}], [\sigma]) & \le d_{as}([\iota_{k_0}], [\iota_{9k_0}]) + \iota_{as}([\iota_{9k_0}], [\iota_{k_1}]) + d_{as}([\iota_{k_1}], [\sigma]) \le 4 \log 3.
\end{align*}
Here we use Lemmas \ref{lemma three} and \ref{lemma asymmetry}. This completes the proof.
\end{proof}

\appendix

\section{Remarks on relation with other notions}\label{other_appen}
\subsection{Quasi-simplicity}
We slightly generalize the concept of quasi-simplicity in the sense of \cite{KM10}.
\begin{definition}
Let $h\colon G\times X \to X$ be an action of a group $G$ on a set $X$.
Then, $h$ is said to be \textit{fixed point free} if $\mathcal{O}(x) \neq \{x\}$
for every $x \in X$, where $\mathcal{O}(x) = \{\,h(g,x) \mid  g \in G\,\}$ is the orbit of $x$.
\end{definition}

\begin{definition}[\cite{KM10}]
A fixed point free action of $G$ on a set $X$ is said to be \textit{quasi-simple} if there is no non-trivial normal subgroup $N$ of $G$ such that the action restricted to $N$ is fixed point free.
\end{definition}

\begin{remark}
    In \cite{KM10}, Kowalik and Michalik considered 
    quasi-simple actions only in the case where $h$ is effective.
    
\end{remark}

By definition, we can easily confirm the following proposition which states that relative simplicity implies quasi-simplicity.
\begin{prop}\label{rs is qs}
    Let $M$ be a smooth manifold and $G$ a subgroup of $\mathrm{Diff}_0^c(M)$.
Assume that the universal covering $\tilde{G}$ is simple relative to $\pi_1(G)$ and the action $h\colon G\times M \to M$ is fixed point free, where $\pi\colon\tilde{G} \to G$ is the universal covering and $h$ is defined by $h(g,x) = \pi(g)(x)$.
Then, the action $h$ is quasi-simple.
\end{prop}
By Proposition \ref{rs is qs}, our results (Theorems \ref{closed_rel_simpl}, \ref{smooth_rel_simpl} and \ref{vol_rel_simpl})
can be regarded as results on quasi-simplicity when the underlying manifold is closed.

\subsection{Maximal normal subgroup}

In this paper, we have mainly discussed the maximum normal subgroup, but it is natural to consider maximal normal subgroups. In this appendix, we describe the basic properties of maximal normal subgroups and discuss their relation with maximum normal subgroups.

\begin{definition}
Let $G$ be a group.
A proper normal subgroup $N$ of $G$ is called a \textit{maximal normal subgroup} if there is no normal subgroup $K$ of $G$ such that $N \subsetneq K \subsetneq G$.
\end{definition}

The following proposition is well known.
\begin{prop}
Let $G$ be a group.
A normal subgroup $N$ of $G$ is a maximal normal subgroup of $G$ if and only if $G/N$ is a simple group.
\end{prop}

By definition, it is clear that the maximum normal subgroup is a maximal normal subgroup.

Here we note that there is a group $G$ such that $G$ has no maximal normal subgroup. For example, the additive group $\QQ$ is such an example.

Moreover, there is a group $G$ such that $G$ has a unique maximal normal subgroup but $G$ is not relatively simple. Such an example is provided by $\QQ \times (\ZZ / p\ZZ )$ for a prime number $p$. 
To see this, since $\ZZ / p\ZZ $ is a simple  group, $\QQ \times \{ 0\}$ is a maximal (normal) subgroup of $\QQ \times (\ZZ / p\ZZ )$. However, $\QQ \times \{ 0\}$ is not maximum since $\{ 0\} \times (\ZZ / p\ZZ )$ is a normal subgroup of $\QQ \times (\ZZ / p\ZZ )$ which is not contained in $\QQ \times \{ 0\}$. 
Finally, we will see that $\QQ \times \{ 0\}$ is the unique maximal normal subgroup of $\QQ \times (\ZZ / p\ZZ )$. Let $N$ be a maximal normal subgroup of $\QQ \times (\ZZ / p\ZZ )$. Then $(\QQ \times (\ZZ / p\ZZ )) / N$ is an abelian simple group and hence is isomorphic to $\ZZ / q\ZZ $ for some prime number $q$. Hence there is a surjective group homomorphism $f \colon \QQ \times (\ZZ / p\ZZ ) \to \ZZ / q\ZZ $ with $\Ker(f) = N$. Then the kernel of $f$ contains $\QQ \times \{ 0\}$. Indeed, for every $x \in \QQ$, we have
\[f(x,0) = q \cdot f(x/q, 0) = 0.\]
Thus we have that $f(x,0) = 0$ for every $x \in \QQ$. Hence $\QQ \times \{ 0\} \subset \Ker(f) = N$. Since $\QQ \times \{ 0\}$ is maximal, we have $\Ker(f) = \QQ \times \{ 0\}$. This completes the proof.

As can be seen from the example of $\QQ$, not all groups have a maximal normal subgroup. This is not similar to the case of maximal ideals in rings. However, as shown below, non-trivial finite normally generated groups have maximal normal subgroups.

\begin{theorem}
Let $G$ be a group having finite normal generators $\{ x_1, \cdots, x_n\}$, and $N$ a proper normal subgroup of $G$. Then there is a maximal normal subgroup of $G$ containing $N$.
\end{theorem}
\begin{proof}
Since $N$ is proper, there is $x_i$ not contained in $N$. By Zorn's lemma there is a normal subgroup $N'$ of $G$ such that $N \subset N'$, $x_i \not\in N'$ and $N'$ is maximal among such normal subgroups. 
If $N'$ is not a maximal normal subgroup of $G$,  then there is a proper normal subgroup $N_1$ of $G$ such that $N' \subsetneq N_1$.
It follows from the maximality of $N'$ that $N_1$ contains $x_i$. Apply the same procedure to $N_1$. Since the number of normal generators is finite, this process will terminate after a finite number of steps, reaching a maximal normal subgroup.
\end{proof}

\bibliographystyle{amsalpha}
\bibliography{Kodama_ref}

\end{document}